\documentclass[12pt]{amsart}
\usepackage{amsfonts}
\usepackage{amssymb}
\usepackage{color}
\usepackage{dsfont}
\usepackage{setspace}
\usepackage[fleqn,tbtags]{mathtools}
\usepackage{microtype}
\numberwithin{equation}{section}
\theoremstyle{plain}
 \newtheorem{theorem}{Theorem}[section]
 \newtheorem{lemma}[theorem]{Lemma}
 \newtheorem{corollary}[theorem]{Corollary}
 \newtheorem{proposition}[theorem]{Proposition}

\theoremstyle{definition}
 \newtheorem{definition}[theorem]{Definition}
 \newtheorem{example}[theorem]{Example}
 \newtheorem{remark}[theorem]{Remark}
\newcommand{\cF}{\mathcal{F}}

\newcommand{\One}{\mathds{1}}

\allowdisplaybreaks

\begin{document}
\author[David Berger]{David Berger$^{*}$}
\thanks{$^{*}$Institute of Mathematical Stochastics, TU Dresden, Zellescher Weg 12--14, 01069 Dresden, Germany.
Email: \texttt{david.berger2@tu-dresden.de}.}

\author[Maximilian Strobel]{Maximilian Strobel$^{\dagger}$}
\thanks{$^{\dagger}$Institute of Mathematical Finance, Ulm University, Helmholtzstr. 18, 89081 Ulm, Germany.
Email: \texttt{maximilian.strobel@uni-ulm.de}.}
\title{Random sampling of almost periodic stochastic processes}
\vspace{5mm}
\begin{abstract}
We study the effect of random sampling on stochastic processes whose finite-dimensional distributions are Bohr almost periodic under time shifts. The sampling times are generated independently by a two-sided random walk with positive increments or by a two-sided
subordinator. Under a non-degeneracy condition we show that the shifted finite-dimensional distributions of the sampled process converge to those of a strictly stationary process. Consequently, if the sampled process is itself almost periodic in distribution, then it must already be strictly stationary. We further prove an $L^2$-law of large numbers for sampled averages under covariance decay and a uniform second-moment bound, and identify its limit with the mean of the stationary limit process. Finally, for $m$-dependent, $L^2$-uniformly integrable almost periodic processes, we establish a central limit theorem for the centered sampled sums, whose asymptotic variance is given by the autocovariance series of the stationary limit.
\end{abstract}
\maketitle
\section{Introduction}
Throughout the paper, $\mathbb{G}$ denotes either $\mathbb{Z}$ or $\mathbb{R}$. We study the effect of random sampling on stochastic processes that are almost periodic in distributional sense. Given a process $X=(X_t)_{t\in \mathbb{R}}$ and an independent random sampling process $Y=(Y_g)_{g\in \mathbb{G}}$, our object of interest is the sampled process $(X_{Y_g})_{g\in \mathbb{G}}$. The basic question is whether almost periodicity survives under random evaluation, and what asymptotic distributional properties emerge after sampling.\\
The classical theory of periodically and almost periodically correlated
processes, introduced by Gladyshev \cite{Gladyshev1963} and developed further by Hurd \cite{Hurd1991}, requires the mean and covariance function to be almost periodic in time. We work with a different distributional formulation, which has been studied and characterised by Berger and Mohamed \cite{BM-22}. We call the process $X$ almost periodic in distributional sense, if for every $n\in\mathbb N$ and every $g_1,\dots,g_n\in \mathbb{R}$, the map
\begin{align*}
    \mathbb{R}\ni g\mapsto \mathcal L(X_{g_1+g},\dots,X_{g_n+g})
\end{align*}
is Bohr almost periodic (c.f. Definition \ref{def_bohr_ap} and \ref{def_bohr_ap_proc}), where $\mathcal{L}(\cdot)$ denotes the law of a random vector. Thus, our stationarisation result concerns the full
finite-dimensional laws, rather than only first- and second-order
characteristics.\\
Random translations have previously been used to obtain stationary processes from nonstationary ones. Hurd \cite{Hurd1974} showed that an independent shift uniformly distributed over one period stationarises a process with periodic finite-dimensional distributions. Iksanov, Kabluchko and Marynych \cite{IksanovKabluchkoMarynych2025} obtained stationary limits from increasingly spread-out uniform translations of deterministic Besicovitch almost periodic functions, including finite-dimensional convergence under additional spectral assumptions.
In our work, the sampling mechanism is different, since a random sampling randomises the spacings between observations, instead of preserving it, as a common random translation would do. In our first main result, we establish a stationarisation mechanism under random sampling of a process which is almost periodic in distributional sense.
More precisely, Theorem \ref{Theorem_stat_Z} shows that, under the assumption that $\mathcal{L}(Y_1)$ is not a lattice distribution, we have
\begin{align*}
     (X_{Y_{g_1+g}},\ldots,X_{Y_{g_n+g}})\underset{|g|\to\infty}{\overset{d}{\longrightarrow}}(Z_{g_1},\ldots,Z_{g_n})
\end{align*}
for every $g_1<\cdots<g_n$ in $\mathbb{G}$, where $Z=(Z_g)_{g\in \mathbb{G}}$ is strictly
stationary. No moment or dependence assumptions on $X$ are needed for this
result. In contrast to stationarisation by a uniformly distributed shift, as done in \cite{Hurd1974}, our conclusion is initially asymptotic. Corollary \ref{cor_ap_stat} identifies when it becomes exact: the sampled process is almost periodic in distributional sense if and only if it is already strictly stationary. Consequently, nonstationary almost periodic behaviour cannot be preserved by random sampling as covered in this corollary.\\
Laws of large numbers for periodically and almost periodically correlated
processes were established by Cambanis, Houdr\'e, Hurd and Le\'skow \cite{CambanisHoudreHurdLeskow1994}. Our Theorem \ref{LLN_sampled} concerns an $L_2$ law of large numbers for the sampled process $(X_{Y_g})_{g\in\mathbb{G}}$. Specifically, write $f_X(t)=\mathbb E(X_t)$ and let $M(f_X)$ denote its Bohr mean (c.f. Lemma \ref{lemma1}). Under a covariance decay condition and a uniform second-moment bound, we obtain
\begin{align*}
    \frac1n\sum_{k=1}^nX_{Y_k}\underset{n\to\infty}{\overset{L_2}{\longrightarrow}}
     M(f_X)=\mathbb E Z_0.
\end{align*}
This result connects averages along the sampling sequence with the stationary limit from Section 3.\\
Mean estimation from random observations has also been studied directly. Dehay and Dudek \cite{DehayDudek2015} studied estimation of the Fourier-Bohr coefficients $\alpha_X(\lambda)=M(t\mapsto e^{-i\lambda t}f_X(t))$ for almost periodically correlated processes observed at the jump times of an independent homogeneous Poisson process. They established asymptotic normality of the estimation error and consistency of a circular block bootstrap procedure. Our law of large numbers is related to this estimation problem: since $M(f_X)$ is the zero-frequency Fourier-Bohr coefficient of the mean function, Theorem \ref{LLN_sampled} establishes mean-square consistency of the sampled average for this coefficient.\\
For second-order characteristics, Monsan \cite{Monsan1994} and Dehay and Monsan \cite{DehayMonsan1996} constructed consistent spectral density estimators from random observations. Dehay and Monsan \cite{DehayMonsan2007} established consistency and asymptotic normality of spectral covariance estimators under jittered periodic sampling, while Dehay and Dudek \cite{DehayDudek2017} studied bootstrap methods for second-order analysis under Poisson sampling. These works concern inference for covariance and spectral characteristics. In this paper, we instead identify the limiting covariance structure of the sampled process and use it to determine the variance in our central limit theorem.\\
Central limit theorems for randomly sampled sums have been established in several stationary settings. Guillotin-Plantard and Prieur \cite{GuillotinPlantardPrieur2010} prove a central limit theorem for weakly dependent sequences sampled by a transient integer-valued random walk. Brandes and Curato [3, Theorem 3.2] proved a central limit theorem for the sample mean of a stationary Lévy-driven moving-average process observed at independent renewal times. Brandes, Curato and Stelzer \cite{BrandesCuratoStelzer2023} obtained inheritance results for strong mixing and weak dependence under renewal sampling of strictly stationary processes. Our Theorem \ref{CLT_sampled} needs the original process to be almost periodic in distribution, but not necessarily stationary. We assume $m$-dependence and $L_2$-uniform integrability, and establish a Gaussian limit for the centered sums. The asymptotic variance is identified through the stationary limit from Section 3. Other works on central limit theorems of processes, which are almost periodic in distributional sense include Berger and Mohamed \cite{BM-22}, who obtained a central limit theorem for unsampled, $L_2$-uniformly integrable, $m$-dependent processes, which are almost periodic in distributional sense.\\
The paper is organized as follows. Section 2 collects the required
almost periodicity and sampling tools. Section 3 establishes the stationary
limit and the consequent rigidity of almost periodicity under sampling.
Section 4 proves the law of large numbers and the central limit theorem for
sampled processes.
\section{Preliminaries}
We first state conventions, notations and definitions, which are used frequently throughout this paper.
\begin{enumerate}
    \item Let $(M,d)$ be a metric space. For $r>0$ and $x\in M$, we denote $B_r(x):=\{y\in M: d(x,y)<r\}$, i.e. the open ball around $x$ with radius $r$. With $\overline{B_r(x)}$ we denote the closure of $B_r(x)$.
    \item For a metric space $(M,d)$ we denoty by $C_b(\mathbb{R},M)$ the space of continuous, bounded functions from $\mathbb{R}$ to $M$, endowed with the supremum norm. With $C(\mathbb{R}, M)$ we denote the space of continuous functions from $\mathbb{R}$ to $M$.
    \item For a random vector $X$ we denote its distribution by $\mathcal{L}(X)$ and its Fourier transform, i.e. its characteristic function, by $\widehat{\mathcal{L}}(X)$, with the convention that $\widehat{\mathcal{L}}(X)(z)=\mathbb{E}(e^{iz^TX})=\int_{\mathbb{R}^n}e^{iz^T x}\, \mathcal{L}(X)(dx)$.
    \item We write $X\overset{d}{=}Y$ for two random vectors $X$ and $Y$ having the same distribution.
    \item Similarly for a sequence of random vectors $(X_n)_{n\in\mathbb{N}}$, we denote convergence in distribution to a random vector $X$ by $X_n\underset{n\to\infty}{\overset{d}{\longrightarrow}}X$. Moreover $X_n\underset{n\to\infty}{\overset{L_p}{\longrightarrow}}X$ denotes convergence in $L_p(P)$ for $p\geq 1$, where $P$ is the underlying probability measure.
    \item We call a family of random variables $(X_t)_{t\in J}$, $J\subset \mathbb{R}$, $L_p$-uniformly integrable, if for all $\varepsilon>0$ there exists a $K>0$ such that
    \begin{align*}
        \sup\limits_{t\in J}\mathbb{E}(|X_t|^p\One_{|X_t|>K})<\varepsilon,
    \end{align*}
    where $\One_{A}$ denotes the indicator function of a set $A$.
    \item A stochastic process $X=(X_t)_{t\in \mathbb{R}}$ is called strictly stationary, if for all $t_1,...,t_n,s\in \mathbb{R}$ it holds
    \begin{align*}
        (X_{t_1},...,X_{t_n})\overset{d}{=}(X_{t_1+s},...,X_{t_n+s}).
    \end{align*}
    \item If not specified otherwise, we will from here on assume that the introduced random variables and stochastic processes are one-dimensional and real valued.  
\end{enumerate}
Let us recall the notion of almost periodic functions.
\begin{definition}\label{def_bohr_ap}
    Let $f:\mathbb{R}\to M$ be continuous, where $(M,d)$ is a complete metric space. We say that $f$ is \emph{almost periodic in the sense of Bohr} if $\{f(\cdot+g):g\in \mathbb{R}\}$ is relatively compact in the space $C_b(\mathbb{R},M)$. We denote the space of functions, which are almost periodic in the sense of Bohr, with $AP(\mathbb{R},M)$. In particular $AP(\mathbb{R}, M)\subset C_b(\mathbb{R}, M)$.
\end{definition}
\noindent
\begin{remark}\label{remark_equiv_bohr_ap}
    \begin{enumerate}
    \item $f\in AP(\mathbb{R},M)$ if and only if $f$ is continuous and for all $\varepsilon>0$ there exists $L_\varepsilon>0$ and $\tau(a,\varepsilon)\in [a,a+L_\varepsilon]$ for all $a\in\mathbb{R}$ such that
    \begin{align*}
        d(f(x+\tau), f(x))<\varepsilon \qquad \forall x\in \mathbb{R}.
    \end{align*}
    For a full proof, see e.g. \cite[Theorem 3.2]{Corduneanu2009}. The above property is also called \emph{almost periodic in the sense of Bochner}.
    \item If $M=\mathbb{C}$, then $AP(\mathbb{R},\mathbb{C})$ is the uniform closure of all $\mathbb{C}$-valued trigonometric polynomials, i.e. 
    \begin{align*}
        AP(\mathbb{R},\mathbb{C})=\overline{\mathcal{T}}^\infty,
    \end{align*}
    with
    \begin{align*}
        \mathcal{T}:=\{f:\mathbb{R}\to\mathbb{C}:f(t)=\sum_{k=1}^n\alpha_ke^{i\lambda_kt} \text{ for some }n\in\mathbb{N}, \alpha_1,...,\alpha_n\in\mathbb{C}, \lambda_1,...,\lambda_n\in\mathbb{R}\},
    \end{align*}
    and $||f||_\infty:=\sup\limits_{t\in\mathbb{R}}|f(t)|$ (see e.g. \cite[Proposition 3.20 and Theorem 4.8]{Corduneanu2009}).\\
    \end{enumerate}
\end{remark}

For $\mathbb{G}=\mathbb{R}$ we denote by $(Y_t)_{t\ge 0}$ a subordinator (i.e. an $\mathbb{R}$-valued Lévy process with increasing sample paths, c.f. \cite[Definition 13.17]{BrockwellLindner2024}) and define a two-sided extension of $Y$ by taking an independent, but identically distributed subordinator $(Y'_{t})_{t\ge 0}$ and setting
\begin{align*}
Y_t:=\begin{cases}
Y_t,\,\textrm{ for } t\ge 0,\\
-Y'_{(-t)-},\,\textrm{ for }t<0,
\end{cases}
\end{align*}
where $Y'_s=\lim\limits_{u\uparrow s}Y'_u$. We then call the new process $Y=(Y_t)_{t\in\mathbb{R}}$ a \emph{random sampling}. If $\mathbb{G}=\mathbb{Z}$, we take a sequence of positive, independent and identically distributed random variables $(\psi_i)_{i\in\mathbb{Z}\setminus\{0\}}$ and define
\begin{align*}
Y_k=\begin{cases}
\sum\limits_{i=1}^k \psi_i,\textrm{ for }k\geq 1,\\
0,\textrm{ for }k=0,\\
\sum\limits_{i=1}^{-k} -\psi_{-i},\textrm{ for }k\leq -1.
\end{cases}
\end{align*}
This defines a \emph{random sampling} in $\mathbb{Z}$.\\
Throughout the remainder of the paper we will assume that a random sampling $Y=(Y_g)_{g\in\mathbb{G}}$ satisfies
\begin{align}\label{eq_lattice_sampling}
    |\widehat{\mathcal{L}}(Y_1)(z)|\ne 1 \qquad \forall z\ne 0.
\end{align}
Note that $|\widehat{\mathcal{L}}(Y_1)(z_0)|=1$ for $z_0\ne 0$ is equivalent to $\mathcal{L}(Y_1)$ being a lattice distribution, i.e. there exists an $a\in\mathbb{R}$ such that $P\left(Y_1\in a+ \frac{2\pi}{|z_0|}\mathbb{Z}\right)=1$ (see e.g. \cite[Theorem 2.1.4]{Lukacs1970}). So our assumption is equivalent to $Y_1$ not being a lattice distribution.\\
We first prove some lemmata, which will mainly be used as tools for the coming sections. The first lemma shows that the sampled expectation asymptotically recovers the Bohr mean $M(f):=\lim\limits_{t\to\infty}\frac{1}{2t}\int_{(-t,t)}f(s)\,ds$ of an almost periodic function $f$. 
\begin{lemma}\label{lemma1}
    Let $f\in AP(\mathbb{R}, \mathbb{C})$. Assume $Y=(Y_g)_{g\in \mathbb{G}}$ is a random sampling such that $|\widehat{\mathcal{L}}(Y_1)(z)|\neq 1$ for all $z\neq 0$ and define
    \begin{align*}
    T_g(f):=\mathbb{E}(f(Y_g)).
    \end{align*}
    Then the limit of $T_g$ for $|g|\to\infty$ exists and is given by
    \begin{align*}
    \lim_{|g|\to\infty}T_g(f)=M(f):=\lim_{t\to\infty}\frac{1}{2t}\int\limits_{(-t,t)} f(s)ds\in\mathbb{R}.
    \end{align*}
\end{lemma}
\begin{proof}
    First note that \cite[Corollary 3.4]{Corduneanu2009} ensures the existence of the Bohr mean $M(f)$ for all $f\in AP(\mathbb{R},\mathbb{C})$. We first look at the case $g\to\infty$.\\
    Let $\varepsilon>0$. Since $f\in AP(\mathbb{R}, \mathbb{C})$, by Remark \ref{remark_equiv_bohr_ap} (ii), there exists $f_{\varepsilon}\in \mathcal{T}$ such that
    \begin{align*}
    \|f-f_{\varepsilon}\|_{\infty}<\varepsilon.
    \end{align*}
    We estimate
    \begin{align*}
    &|T_g(f)-M(f)|\\
    \leq&|T_g(f)-T_g(f_{\varepsilon})|+|T_g(f_{\varepsilon})-M(f_\varepsilon)|+|M(f_\varepsilon)-M(f)|\\
    \leq& \|f-f_{\varepsilon}\|_{\infty}+|T_g(f_{\varepsilon})-M(f_\varepsilon)|+\limsup_{t\to\infty}\frac{1}{2t}\int_{-t}^t|f(s)-f_{\varepsilon}(s)|ds\\
    \leq& 2\|f-f_{\varepsilon}\|_{\infty}+|T_g(f_{\varepsilon})-M(f_\varepsilon)|,
    \end{align*}
    so it suffices to show that $|T_g(f_{\varepsilon})-M(f_\varepsilon)|\longrightarrow 0$ as $g\to\infty$. Let $f_{\varepsilon}(x)=\sum\limits_{k=0}^n \alpha_{k} e^{i\lambda_k x}$ with $\lambda_k\in \mathbb{R}$, $a_k\in \mathbb{C}$ for all $k=0,\dotso,n\in \mathbb{N}$, $\lambda_0=0$ and $\lambda_k\ne 0$ for all $k\ne 0$. It clearly holds that
    \begin{align*}
    M(f_\varepsilon)=\alpha_0 
    \end{align*} 
    and
    \begin{align*}
    T_g(f_{\varepsilon})
    =\alpha_0+\sum\limits_{k=1}^n \alpha_k \mathbb{E}e^{i\lambda_kY_g}=\alpha_0+\sum\limits_{k=1}^n \alpha_k (\mathbb{E}e^{i\lambda_kY_1})^g;
    \end{align*}
    here, the second equality is clear when $\mathbb{G}=\mathbb{Z}$ and follows from the Lévy-Khintchine formula for Lévy processes when $\mathbb{G}=\mathbb{R}$ with $(\mathbb{E}(e^{iY_1\lambda_k}))^g$ being the distinguished $g$'th power of $(\mathbb{E}(e^{iY_1\cdot}))^g$, evaluated at $\lambda_k$, see e.g. \cite[Definition 10.24 and Theorem 10.26]{BrockwellLindner2024}.
    Since we assumed that $|\mathbb{E}e^{izY_1}|=|\widehat{\mathcal{L}}(Y_1)(z)|\ne 1$ for all $z\ne 0$, we obtain that
    \begin{align*}
    |T_g(f_\varepsilon)-\alpha_0|\leq  \sum\limits_{k=1}^n |\alpha_k||\mathbb{E}e^{i\lambda_kY_1}|^g\underset{g\to\infty}{\longrightarrow}0.
    \end{align*}   
    Now let $g<0$, then $Y_g=-Y_{-g}'$ a.s. Since $|\widehat{\mathcal{L}}(-Y_{-g}')(z)|=|\widehat{\mathcal{L}}(Y_{-g})(-z)|\ne 1$ for all $z\ne 0$, we can apply what we have already proven and get
    \begin{align*}
        \lim\limits_{g\to\infty}T_{-g}(f)=\lim\limits_{g\to\infty}\mathbb{E}(f(-Y_{g}'))=M(f(-\cdot))=M(f),
    \end{align*}
    which completes the proof.
\end{proof}
\noindent
It is worth noting that this result is not true for deterministic sampling. If we take $Y_g=ag$ for some $a\ne 0$ then the lemma already fails for monomials of the kind $f(x)=e^{i\lambda x}$ for some $\lambda >0$, since $M(f)=0$ and $|\mathbb{E}(f(Y_g))|= 1$ for all $g\in\mathbb{R}$. As we saw in the previous lemma, random sampling with the extra condition $|\widehat{\mathcal{L}}(Y_1)(z)|\neq 1$ for all $z\neq 0$ changes this behaviour. We now prove that this behaviour even holds true uniformly over relatively compact families of almost periodic functions.

\begin{lemma}\label{lemma_1_sup}
    Let $F\subset AP(\mathbb{R},\mathbb{C})$ be a relatively compact set with respect to $||\cdot||_\infty$. Let further $Y=(Y_g)_{g\in\mathbb{G}}$ be a random sampling with $|\widehat{\mathcal{L}}(Y_1)(z)|\ne 1$ for all $z\ne 0$. Then
    \begin{align*}
        \lim\limits_{|g|\to\infty} \sup\limits_{f\in F}|\mathbb{E}f(Y_g)-M(f)|=0.
    \end{align*}
\end{lemma}
\begin{proof}
    Let $\varepsilon>0$. Since $F$ is relatively compact, for each $\varepsilon>0$ we can find functions $f_1,...,f_n\in AP(\mathbb{R},\mathbb{C})$ such that $F\subset \bigcup_{k=1}^n B_\varepsilon(f_k)$. For $f\in F$, we therefore know that there exists $j\in\{1,...,n\}$ such that $||f-f_j||_\infty<\varepsilon$. We estimate
    \begin{align*}
        |\mathbb{E}f(Y_g)-M(f)|\leq& |\mathbb{E}f(Y_g)-\mathbb{E}f_j(Y_g)|+|\mathbb{E}f_j(Y_g)-M(f_j)|+|M(f_j)-M(f)|\\
        \leq & 2 ||f-f_j||_\infty + \max\limits_{k\in\{1,...,n\}}|\mathbb{E}f_k(Y_g)-M(f_k)|\\
        \leq & 2\varepsilon + \max\limits_{k\in\{1,...,n\}}|\mathbb{E}f_k(Y_g)-M(f_k)|,
    \end{align*}
    and thus also
    \begin{align*}
        \sup\limits_{f\in F}|\mathbb{E}f(Y_g)-M(f)|\leq 2\varepsilon +\max\limits_{k\in\{1,...,n\}}|\mathbb{E}f_k(Y_g)-M(f_k)|.
    \end{align*}
    Lemma \ref{lemma1} gives $\lim\limits_{|g|\to\infty}\max\limits_{k\in\{1,...,n\}}|\mathbb{E}f_k(Y_g)-M(f_k)|=0$ and thus the proof is complete.
\end{proof}

\begin{corollary}
\label{lemma_uniformity}
Let $f,h\in AP(\mathbb{R},\mathbb{C})$ and assume $Y=(Y_g)_{g\in \mathbb{G}}$ is a random sampling such that $|\widehat{\mathcal{L}}(Y_1)(z)|\neq 1$ for all $z\neq 0$. Then
\begin{align*}
    \lim_{|g|\to \infty} \sup_{x\in \mathbb{R}} |\mathbb{E}f(x+Y_g)-M(f)|=0
\end{align*}
and
\begin{align*}
    \lim_{|g|\to\infty} \sup_{x,y\in \mathbb{R}} |\mathbb{E}f(x+Y_g)h(y+Y_g)-M(f(\cdot+x)h(\cdot+y))|=0.
\end{align*}
\end{corollary}
\begin{proof}
    If $f,h\in AP(\mathbb{R},\mathbb{C})$, then also $fh\in AP(\mathbb{R},\mathbb{C})$, see e.g. \cite[Remark 3.4.IV]{Corduneanu2009}. Moreover by definition the sets $\{f(\cdot +x): x\in\mathbb{R}\}$ and $\{h(\cdot+x):x\in\mathbb{R}\}$ are relatively compact and so is the set $\{f(\cdot +x)h(\cdot +y):x,y\in\mathbb{R}\}$. Since $M(f)=M(f(x+\cdot))$, the assertion follows immediately with Lemma \ref{lemma_1_sup}.
\end{proof}
\begin{remark}
Let $\check{h}(t):=h(-t)$. For $s=x-y$, translation invariance of the
Bohr mean gives
\[
    M(f(\cdot+x)h(\cdot+y))
    =M(f(\cdot+s)h(\cdot))
    =(f\circledast\check h)(s),
\]
where $\circledast$ denotes the Eberlein convolution, see
\cite[Section 3]{LenzSpindelerStrungaru}. Thus, Corollary 2.5 shows that random sampling asymptotically recovers the Eberlein convolution, uniformly in $x$ and $y$.
\end{remark}
\begin{definition}\label{def_bohr_ap_proc}
Let $X=(X_t)_{t\in\mathbb R}$ be a real-valued stochastic process.
We say that $X$ is \emph{almost periodic in distributional sense} if, for every
    $n\in\mathbb N$ and every $t_1,\ldots,t_n\in\mathbb R$, the map
    \begin{align*}
        s\mapsto \mathcal L(X_{t_1+s},\ldots,X_{t_n+s})
    \end{align*}
    is Bohr almost periodic as a map from $\mathbb R$ into the complete metric space $(\mathcal P(\mathbb R^n),d_{\mathrm{LP}})$, where $\mathcal{P}(\mathbb{R}^n)$ denotes the space of probability measures on $\mathbb{R}^n$ and $d_{LP}$ denotes the Lévy-Prokhorov metric.
\end{definition} 
\noindent
In \cite{BM-22} almost periodic processes in distributional sense are also called almost periodic stationary and several properties of such processes and equivalent definitions in terms of characteristic functions are given. \\
It is clear that every strictly stationary process is almost periodic in distributional sense, but the converse is not necessarily true as the deterministic process given by $X_t=\cos(t)$ already shows.\\
In further sections we look at processes which are almost periodic in distributional sense as well as subordinated processes. Since the assumption that $X$ is almost periodic does not guarantee that $X_{Y_g}$ is measurable, we from now on assume that $(t,\omega)\mapsto X_t(\omega)$ is jointly measurable.\\
\noindent
The next lemma is a preparatory lemma for the main results in Section 3. 
\begin{lemma}\label{rel_compact_char}
    Let $X=(X_t)_{t\in\mathbb R}$ be a real-valued stochastic process, which is almost periodic in distributional sense. Let $n\in\mathbb{N}$ and  $u=(u_1,...,u_{n-1})\in [0,\infty)^{n-1}$. Define for $z\in\mathbb{R}^{n}$
    \begin{align*}
        F_{u,z}(g):=\widehat{\mathcal{L}}(X_g,X_{g+u_1},...,X_{g+\sum_{k=1}^{n-1}u_k})(z), \qquad g\in\mathbb{R}.
    \end{align*}
    Then for every $r>0$ the set
    \begin{align*}
        H_{u,r}:=\{F_{u,z}: z\in \overline{B_r(0)}\}
    \end{align*}
    is compact in $(AP(\mathbb{R},\mathbb{C}), ||\cdot||_\infty)$.
\end{lemma}
\begin{proof}
    Since $X$ is almost periodic in distributional sense, we get for each $u\in [0,\infty)^{n-1}$ and $z\in \overline{B_r(0)}$ that the function $F_{u,z}$ is in $AP(\mathbb{R},\mathbb{C})$ and thus $H_{u,z}\subset AP(\mathbb{R},\mathbb{C})$; this follows from \cite[Theorem 2.9]{BM-22}, applied with $K:=\{z\}$.\\
    For $u\in[0,\infty)^{n-1}$ we define $G_u: \mathbb{R}\to C(\overline{B_r(0)}, \mathbb{C})$ by 
    \begin{align*}
        G_u(g)(z):=\widehat{\mathcal{L}}(X_g,X_{g+u_1},...,X_{g+\sum_{k=1}^{n-1}u_k})(z)=F_{u,z}(g).
    \end{align*}
    By \cite[Theorem 2.9]{BM-22}, $g\mapsto G_u(g)(z)$ is continuous for each $z\in \mathbb{R}^n$. Let now $g_n\to g$. Since $G_u(g)(z)=F_{u,z}(g)$ is a characteristic function and $G_u(g_n)$ converges pointwise to $G_u(g)$, Lévy's continuity theorem gives that $G_u(g_n)(z)\underset{n\to\infty}{\longrightarrow}G_u(g)(z)$ locally uniformly in $z$, so $G_u\in C(\mathbb{R}, C(\overline{B_r(0)}, \mathbb{C}))$. Applying \cite[Theorem 2.9]{BM-22} with $K:=\overline{B_r(0)}$ furthermore gives that $G_u$ is almost periodic in the sense of Bochner and thus by Remark \ref{remark_equiv_bohr_ap}(i), $G_u\in AP(\mathbb{R}, C(\overline{B_r(0)}, \mathbb{C}))$.
    \cite[Proposition 3.18]{Corduneanu2009} now shows that $\widetilde{H}_{u,r}:=\{G_u(g): g\in\mathbb{R}\}$ is relatively compact in $C(\overline{B_r(0)}, \mathbb{C})$ and the Theorem of Arzelà-Ascoli thus yields
    \begin{align*}
        \sup\limits_{f\in \widetilde{H}_{u,r}}|f(z)-f(z')|\underset{z\to z'}{\longrightarrow}0.
    \end{align*}
    Hence for $z,z'\in \overline{B_r(0)}$ we get
    \begin{align*}
        \sup\limits_{g\in\mathbb{R}}|F_{u,z}(g)-F_{u,z'}(g)|=\sup\limits_{g\in\mathbb{R}}|G_u(g)(z)-G_u(g)(z')|\leq \sup\limits_{f\in \widetilde{H}_{u,r}}|f(z)-f(z')|,
    \end{align*}
    yielding that $z\mapsto F_{u,z}$ is a continuous mapping from $\overline{B_r(0)}$ to $(C_b(\mathbb{R},\mathbb{C}),||\cdot||_\infty)$. This gives that $H_{u,r}$ is compact in $(AP(\mathbb{R},\mathbb{C}),||\cdot||_\infty)$.     
\end{proof}

\section{Stationary limits under random sampling}
We now study the asymptotic effect of random sampling on processes which are almost periodic in distributional sense. Our main result is that, under the non-degeneracy condition $|\widetilde{\mathcal{L}}(Y_1)(z)|\ne 1$ for all $z\ne 0$, the shifted finite dimensional distributions of the sampled process converge to those of a strictly stationary limit process. As a consequence, a sampled process, which is almost periodic in distribution, must necessarily be strictly stationary. 

\begin{theorem}\label{Theorem_stat_Z}
Let $(X_t)_{t\in \mathbb{R}}$ be almost periodic in distributional sense, $Y$ be a random sampling, independent of $X$, such that $|\widehat{\mathcal{L}}(Y_1)(z)|\neq 1$ for all $z\neq 0$ and $g_1<g_2<\dotso<g_n\in \mathbb{G}$ for some $n\in\mathbb{N}$. 
Then, 
\begin{align*}
    (X_{Y_{g_1+g}},\ldots,X_{Y_{g_n+g}})
\underset{|g|\to\infty}{\overset{d}{\longrightarrow}} Z^{g_1,\ldots,g_n}.    
\end{align*}
Moreover, the law of $Z^{g_1,\ldots,g_n}$ depends only on the increments
$g_2-g_1,\ldots,g_n-g_{n-1}$. The characteristic function of $Z^{g_1,\ldots,g_n}$ at $z\in\mathbb{R}^n$ is given by
\begin{align*}
\lim\limits_{t\to\infty}\frac{1}{2t}\int\limits_{(-t,t)}\int\limits_{[0,\infty)}\dotso\int\limits_{[0,\infty)}\widehat{\mathcal{L}}(X_{s},\dotso,X_{s+\sum\limits_{i=1}^{n-1}u_i})(z)\mathcal{L}(Y_{g_2-g_1})(du_1)\dotso\mathcal{L}(Y_{g_n-g_{n-1}})(du_{n-1})ds.
\end{align*}
\end{theorem}
\begin{proof}
We first treat the case $g\to\infty$. Let $n\in\mathbb{N}$ and $g_1<g_2<...<g_n \in \mathbb{G}$. Assume $g$ to be so large, that $g_1+g>0$.  As $0<g_1+g<g_2+g<\dotso<g_n+g$, we write for $j\in\{2,...,n\}$, $g_j+g=g_1+g+\sum\limits_{k=1}^{j-1} ((g_{k+1}+g)-(g_k+g))$. Since $Y$ is a random sampling, we know that $Y_{g_{j+1}+g}-Y_{g_j+g}\overset{d}{=}Y_{g_{j+1}-g_j}$ and that $Y_{g_1+g},(Y_{g_2+g}-Y_{g_1+g}),...,(Y_{g_n+g}-Y_{g_{n-1}+g})$ are independent. Furthermore we have
\begin{align*}
    Y_{g_j+g}=Y_{g_1+g}+\sum_{k=1}^{j-1}(Y_{g_{k+1}+g}-Y_{g_k+g}).
\end{align*}
Using Fubini's theorem and the independence of the increments of $Y$, as well as the independence of $X$ and $Y$,  we write
\begin{align*}
    &\widehat{\mathcal{L}}(X_{Y_{g_1+g}},\dotso,X_{Y_{g_n+g}})(z)=\widehat{\mathcal{L}}(X_{Y_{g_1+g}},\dotso,X_{Y_{g_1+g}+\sum_{k=1}^{n-1}Y_{g_{k+1}+g}-Y_{g_k+g}})(z)\\
    =&\int_{[0,\infty)}...\int_{[0,\infty)}\mathbb{E}(F_{u,z}(Y_{g_1+g}))\mathcal{L}(Y_{g_2-g_1})(du_1)\dotso\mathcal{L}(Y_{g_n-g_{n-1}})(du_{n-1}), 
\end{align*}
with $F_{u,z}$ as in Lemma \ref{rel_compact_char}. Now fix $r>0$. With Lemma \ref{rel_compact_char} we know that the set $H_{u,r}=\{F_{u,z}: z\in \overline{B_r(0)}\}$ is compact in $(AP(\mathbb{R},\mathbb{C}), ||\cdot||_\infty)$, so for every fixed $u\in [0,\infty)^{n-1}$, Lemma \ref{lemma_1_sup} gives that 
\begin{align*}
    \sup\limits_{z\in \overline{B_r(0)}}|\mathbb{E}(F_{u,z}(Y_{g_1+g})) - M(F_{u,z})|\underset{g\to\infty}{\longrightarrow}0,
\end{align*}
with $M(F_{u,z})=\lim\limits_{t\to\infty}\frac{1}{2t}\int_{(-t,t)}\widehat{\mathcal{L}}(X_s, X_{s+u_1},...,X_{s+\sum_{k=1}^{n-1}u_k})(z)\,ds$. We estimate
\begin{align*}
    &\sup\limits_{z\in \overline{B_r(0)}}\Bigg|\int_{[0,\infty)}...\int_{[0,\infty)}\mathbb{E}(F_{u,z}(Y_{g_1+g}))\mathcal{L}(Y_{g_2-g_1})(du_1)\dotso\mathcal{L}(Y_{g_n-g_{n-1}})(du_{n-1})\\
    &-\int_{[0,\infty)}...\int_{[0,\infty)}M(F_{u,z})\mathcal{L}(Y_{g_2-g_1})(du_1)\dotso\mathcal{L}(Y_{g_n-g_{n-1}})(du_{n-1})\Bigg| \\
    &\leq \int_{[0,\infty)}...\int_{[0,\infty)} \sup\limits_{z\in\overline{B_r(0)}}|\mathbb{E}(F_{u,z}(Y_{g_1+g})) - M(F_{u,z})|\mathcal{L}(Y_{g_2-g_1})(du_1)\dotso\mathcal{L}(Y_{g_n-g_{n-1}})(du_{n-1})
\end{align*}
and since $\sup\limits_{z\in \overline{B_r(0)}}|\mathbb{E}(F_{u,z}(Y_{g_1+g})) - M(F_{u,z})|\leq 2$ an application of Lebesgue's dominated convergence theorem shows that the right hand side of the inequality converges to $0$ as $g\to\infty$. Thus the sequence of characteristic functions $\widehat{\mathcal{L}}(X_{Y_{g_1+g}},\dotso,X_{Y_{g_n+g}})$ converges locally uniformly to $\int_{[0,\infty)}...\int_{[0,\infty)}M(F_{u,z})\mathcal{L}(Y_{g_2-g_1})(du_1)\dotso\mathcal{L}(Y_{g_n-g_{n-1}})(du_{n-1})$. Since the convergence is locally uniform we get that the limit function has to be continuous in $0$ and Lévy's continuity theorem now gives that
\begin{align*}
    z\mapsto\int_{[0,\infty)}...\int_{[0,\infty)}M(F_{u,z})\mathcal{L}(Y_{g_2-g_1})(du_1)\dotso\mathcal{L}(Y_{g_n-g_{n-1}})(du_{n-1})
\end{align*}
is a characteristic function of a random vector $Z^{g_1,...,g_n}$ only depending on the size of the increments $g_2-g_1,...,g_n-g_{n-1}$. The exact form of the characteristic function as stated in the theorem can now be easily obtained by another application of Lebesgue's dominated convergence theorem and Fubini's theorem.\\
Now to the case $g\to-\infty$. Let again $n\in\mathbb{N}$ and $g_1<...<g_n\in \mathbb{G}$. We assume $g$ to be so small, that $g_n+g<0$. We define $q_k:=-(g_{k}+g)$ for $k\in{1,...,n}$. This implies $0<q_n<...<q_1$. Furthermore the variables $(Y_{g_{j+1}+g}-Y_{g_j+g})$ are independent for $j\in\{1,...,n-1\}$ and are independent of $Y_{g_n+g}$. We write
\begin{align*}
    Y_{g_j+g}=Y_{g_n+g}-\sum_{k=j}^{n-1}(Y_{g_{k+1}+g}-Y_{g_k+g})
\end{align*}
and define for $u\in [0,\infty)^{n-1}, z\in\mathbb{R}^n$, $\widetilde{F}_{u,z}(s):=\widehat{\mathcal{L}}(X_{s-\sum_{k=1}^{n-1}u_k},X_{s-\sum_{k=2}^{n-1}u_k},...,X_{s})$. Then $\widetilde{F}_{u,z}(s)=F_{u,z}(s-\sum_{k=1}^{n-1}u_k)$ and since the Bohr mean is translation invariant we obtain $M(\widetilde{F}_{u,z})=M(F_{u,z})$. Since also $\widetilde{H}_{u,r}=\{\widetilde{F}_{u,z}:z\in\overline{B_r(0)}\}$ is compact in $(AP(\mathbb{R},\mathbb{C}), ||\cdot||_\infty)$, we get with Lemma \ref{lemma_1_sup} that
\begin{align*}
    \sup\limits_{z\in \overline{B_r(0)}}|\mathbb{E}(\widetilde{F}_{u,z}(Y_{g_n+g})) - M(F_{u,z})|\underset{g\to-\infty}{\longrightarrow}0.
\end{align*}
The rest of the proof follows verbatim to the positive case.
\end{proof}
\begin{remark}\label{remark_after_stat_thrm}
\begin{enumerate}
    \item The condition $|\widehat{\mathcal{L}}(Y_1)(z)|\ne 1$ for all $z\ne 0$ cannot simply be omitted, as the later Example \ref{exmp_non_ap_stat} shows.
    \item $Z^{g_1,...,g_n}$ can be seen as a collection of finite-dimensional limits, which only depend on the increments.
    The family of limiting finite-dimensional distributions is consistent, since it is obtained as the weak limit of consistent finite-dimensional distributions. By Kolmogorov's extension theorem it therefore defines a process
    $Z=(Z_g)_{g\in G}$. Since the limiting laws depend only on the increments
    between the indices, $Z$ is strictly stationary.
    \item Since weak convergence and uniform $L_1$ integrability imply convergence of the first moment (see e.g. \cite[Theorem 25.12]{Billingsley2012}) we get: If $X$ and $Y$ satisfy the assumptions of Theorem \ref{Theorem_stat_Z} and if $X$ is additionally $L_1$-uniformly integrable, then it is easily checked that also $(X_{Y_g})_{g\in\mathbb{G}}$ is $L_1$-uniformly integrable, hence
    \begin{align*}
        \lim\limits_{|g|\to\infty}\mathbb{E}(X_{Y_g})=\mathbb{E}(Z_0).
    \end{align*}
    If $X$ is even $L_2$-uniformly integrable, then also
    \begin{align*}
         \lim\limits_{|g|\to\infty} {\rm{Cov}}(X_{Y_{g+h}}, X_{Y_g})={\rm{Cov}}(Z_h,Z_0),
    \end{align*}
    for each fixed $h\in\mathbb{G}$.
\end{enumerate}

\end{remark}
\noindent
As a consequence of Theorem \ref{Theorem_stat_Z} we show that $X_Y$ can only be almost periodic if it is already strictly stationary. We only state it for $\mathbb{G}=\mathbb{R}$. The statement can also be shown verbatim for $\mathbb{G}=\mathbb{Z}$, but it would require to define the notion of stochastic processes which are almost periodic in distributional sense on a more general level, which we have decided to omit to improve the readability of the paper.
\begin{corollary}\label{cor_ap_stat}
Let $X=(X_t)_{t\in \mathbb{R}}$ be almost periodic in distributional sense and $Y=(Y_g)_{g\in \mathbb{R}}$ be a random sampling, independent of $X$, such that $|\widehat{\mathcal{L}}(Y_1)(z)|\neq 1$ for all $z\neq 0$. Then the following assertions are equivalent:
\begin{itemize}
\item[1)] $X_Y$ is almost periodic in distributional sense.
\item[2)] $X_Y$ is strictly stationary.
\end{itemize}
\end{corollary}
\begin{proof}
It is clear that $2)$ implies $1)$, so we show that $1)$ implies $2)$. Let $X_Y$ be almost periodic in distributional sense. With Theorem \ref{Theorem_stat_Z} we get that there exists $\rho\in\mathcal{P}(\mathbb{R})$ such that $d_{LP}(\mathcal{L}(X_{Y_g}), \rho)\underset{g\to\infty}{\longrightarrow} 0$. Define $f:\mathbb{R}\to \mathcal{P}(\mathbb{R})$ by $f(g)=\mathcal{L}(X_{Y_g})$, then, by assumption, $f$ is almost periodic on $(\mathcal{P}(\mathbb{R}), d_{LP})$.By Definition \ref{def_bohr_ap} we find for every sequence $(g_n)_{n\in\mathbb{N}}\subset \mathbb{R}$ a subsequence $(g_{n_k})_{k\in\mathbb{N}}$ such that $\mathcal{L}(X_{Y_{g_{n_k}+\cdot}})$ converges in $C(\mathbb{R}, \mathcal{P}(\mathbb{R}))$ to a function $\varphi$, i.e. 
\begin{align*}
    \lim\limits_{k\to\infty}\sup\limits_{h\in\mathbb{R}}d_{LP}(\mathcal{L}(X_{Y_{g_{n_k}+h}}), \varphi(h))=0.
\end{align*}
Assume now that $g_n\to\infty$ as $n\to\infty$. Then for each fixed $h\in\mathbb{R}$ we have $h+g_{{n_k}}\to\infty$ as $k\to\infty$, so with Theorem \ref{Theorem_stat_Z} we have $\lim\limits_{k\to\infty}d_{LP}(\mathcal{L}(X_{Y_{g_{n_k}+h}}), \rho)=0$. Hence $\rho=\varphi(h)$ in $\mathcal{P}(\mathbb{R})$ for all $h\in\mathbb{R}$. Now let $g\in\mathbb{R}$, then
\begin{align*}
    d_{LP}(\mathcal{L}(X_{Y_g}),\rho)=d_{LP}(\mathcal{L}(X_{Y_{(g-g_{n_k})+g_{n_k}}}),\rho)\leq \sup_{h\in \mathbb{\mathbb{R}}}d_{LP}(\mathcal{L}(X_{Y_{h+g_{n_k}}}),\rho)\underset{k\to\infty}{\longrightarrow} 0.
\end{align*}
Hence, $d_{LP}(\mathcal{L}(X_{Y_g}),\rho)=0$ for all $g$, so all 1-dimensional marginals of $X_Y$ coincide. The same argument can be applied to all finite-dimensional marginals and thus we conclude that the process $X_Y$ is strictly stationary.
\end{proof}
\noindent
We now give an example which shows that the non-degeneracy condition regarding the subordinator $Y$ in Theorem \ref{Theorem_stat_Z} and thus also Corollary \ref{cor_ap_stat} cannot simply be omitted. Furthermore we give an example of a stochastic process which is almost periodic, but the subordinated process is not almost periodic for any subordinator $Y$, satisfying $|\widehat{\mathcal{L}}(Y_1)(z)|\ne 1$ for all $z\ne 0$. 
\begin{example}\label{exmp_non_ap_stat}
\begin{enumerate}
    \item Let $X_t=\cos(t)$, then the process defined by $X=(X_t)_{t\in\mathbb{R}}$ is $2\pi$-periodic, hence almost periodic in distributional sense, but not strictly stationary. Define the subordinator $Y$ by $Y_g=g$ for $g\in\mathbb{R}$, then $|\mathcal{L}(Y_1)(z)|=1$ for all $z\in\mathbb{R}$ and $X_Y=X$ is almost periodic in distributional sense, but not strictly stationary. Letting $g\to\infty$, we see that $X_{Y_g}$ cannot converge in distribution, since $X_{Y_{2\pi k}}=1$ and $X_{Y_{(2k+1)\pi}}=-1$ for all $k\in\mathbb{N}$.
    \item Let $X$ be as in the previous example. Let $Y$ be any subordinator that satisfies $|\widehat{\mathcal{L}}(Y_1)(z)|\ne 1$ for all $z\ne 0$. Assume $X_Y$ is almost periodic in distributional sense, then the function $f:\mathbb{R}\to \mathbb{R}$ defined by 
    \begin{align*}
        f(g):=\mathbb{E}(X_{Y_g})=\mathbb{E}(\cos(Y_g))
    \end{align*}
    is almost periodic, since the process is $L_1$-uniformly integrable (c.f. \cite[Remark 2.7]{BM-22}). Furthermore by Lemma \ref{lemma1},
    \begin{align*}
        f(g)=\mathbb{E}(\cos(Y_g))\underset{g\to\infty}{\longrightarrow}M(\cos(\cdot))=0.
    \end{align*}
    However if an almost periodic function has a limit for $g\to\infty$, then it has to be constant, yielding that $f(g)=0$ for all $g\in\mathbb{R}$. To see this choose any sequence $(g_n)_{n\in\mathbb{N}}$ such that $g_n\to\infty$. By the definition of almost periodicity, there exists a subsequence $(g_{n_k})_{k\in\mathbb{N}}$ and a function $h$, such that $f(\cdot+g_{n_k})\underset{k\to\infty}{\longrightarrow}h$ uniformly. Since we also have $f(g+g_{n_k})\underset{k\to\infty}{\longrightarrow}0$, we conclude $h(g)=0$ for all $g\in\mathbb{R}$. Moreover, since $||f||_\infty=||f(\cdot+g_{n_k})||_\infty\underset{k\to\infty}{\longrightarrow}0$, we conclude $f(g)=0$ for all $g\in\mathbb{R}$. However $f(0)=\mathbb{E}(\cos(Y_0))=1$ and thus the assumption that $X_Y$ is almost periodic in distributional sense cannot hold.
\end{enumerate}
\end{example}
\section{Discrete-time asymptotics of sampled processes}
We now look at the discrete-time asymptotic behaviour of the sampled process $(X_{Y_k})_{k\in\mathbb{Z}}$. We derive an $L_2$-law of large numbers for sampled almost periodic processes and later in the section a central limit theorem for sampled $m$-dependent almost periodic processes, with the limiting mean and variance naturally expressed in terms of the stationary limit process constructed in section 3.\\
First we recall a theorem which is due to Robbins \cite[Theorem 1 and Theorem 3]{Robbins1953}.
\begin{theorem}(Robbins)\label{thrm_robbins}
    Let $(\psi_n)_{n\in\mathbb{N}}$ be a sequence of i.i.d. random variables and $S_n:=\sum_{k=1}^n\psi_k$. Let further $\widehat{\mathcal{L}}(\psi_1)(z)\ne 1$ for all $z\ne 0$, then for $f\in AP(\mathbb{R},\mathbb{C})$ 
    \begin{align*}
        \lim\limits_{n\to\infty}\frac{1}{n}\sum_{k=1}^n f(S_k) \underset{n\to\infty}{\overset{a.s.}{\longrightarrow}}M(f),
    \end{align*}
    with $M(f)$ being the Bohr mean, defined in Lemma \ref{lemma1}.\\
    If $(Y_t)_{t\geq 0}$ is a stochastic process with stationary increments such that for every $z\ne 0$ there exists a $\delta>0$ for which
    \begin{align*}
        \widehat{\mathcal{L}}(Y_t-Y_0)(z)=\mathcal{O}(t^{-\delta}),\qquad \text{ as }t\to\infty,
    \end{align*}
    then for $f\in AP(\mathbb{R},\mathbb{C})$
    \begin{align*}
        \lim\limits_{T\to\infty}\frac{1}{T}\int_0^Tf(Y_t)\,dt\underset{T\to\infty}{\overset{a.s.}{\longrightarrow}}M(f).
    \end{align*}
\end{theorem}
\noindent
If we assume $(Y_g)_{g\in\mathbb{Z}}$ to be a random sampling with $|\widehat{\mathcal{L}}(Y_1)(z)|<1$ for all $z\ne 0$, then the assumptions of Theorem \ref{thrm_robbins} are satisfied. Since also $\left|\frac{1}{n}\sum_{k=1}^nf(Y_k)-M(f)\right|^2\leq 4||f||_\infty^2<\infty$, Lebesgue's dominated convergence theorem gives
\begin{align*}
    \lim\limits_{n\to\infty}\frac{1}{n}\sum_{k=1}^n f(Y_k) \underset{n\to\infty}{\overset{L_2}{\longrightarrow}}M(f).
\end{align*}
Similarly one can show that in the case of $\mathbb{G}=\mathbb{R}$ it holds
\begin{align*}
    \lim\limits_{T\to\infty}\frac{1}{T}\int_0^Tf(Y_t)\,dt\underset{T\to\infty}{\overset{L_2}{\longrightarrow}}M(f).
\end{align*}
Since for $g<0$, $f(Y_g)=f(-Y'_{-g})$ a.s. and $M(f(-\cdot)=M(f)$, we can also derive
\begin{align*}
    \lim\limits_{n\to\infty}\frac{1}{n}\sum_{k=1}^n f(Y_{-k}) \underset{n\to\infty}{\overset{L_2}{\longrightarrow}}M(f),\quad \text{ as well as } \quad \lim\limits_{T\to\infty}\frac{1}{T}\int_0^Tf(Y_{-t})\,dt\underset{T\to\infty}{\overset{L_2}{\longrightarrow}}M(f).
\end{align*}
\noindent
With these $L_2$-limits in place, we are now able to prove an $L_2$-Law of Large Numbers for the sampled process $X_{Y}$.

\begin{theorem}
\label{LLN_sampled}
    Let $(X_t)_{t\in\mathbb{R}}$ be a stochastic process, which is almost periodic in distributional sense, with $a_X(h):= \sup\limits_{t\in\mathbb{R}}|\rm{Cov}(X_t,X_{t+h})|\underset{h\to\infty}{\longrightarrow}0$, and $\sup\limits_{t\in \mathbb{R}}\mathbb{E}X_t^2<\infty$. Let $(Y_g)_{g\in\mathbb{Z}}$ be a random sampling, independent of $X$, such that $|\widehat{\mathcal{L}}(Y_1)(z)|\ne 1$ for all $z\ne 0$. Define $f_X(t):=\mathbb{E}(X_t)$. Then
    \begin{align*}
        \frac1n \sum_{k=1}^n X_{Y_k}\underset{n\to\infty}{\overset{L_2}{\longrightarrow}}M(f_X), \text{ as well as } \frac1n \sum_{k=1}^n X_{Y_{-k}}\underset{n\to\infty}{\overset{L_2}{\longrightarrow}}M(f_X).
    \end{align*}
    Denoting $Z=(Z_g)_{g\in\mathbb{Z}}$ the process obtained in Theorem \ref{Theorem_stat_Z}, this can be rewritten as
    \begin{align*}
        \frac1n \sum_{k=1}^n X_{Y_k}\underset{n\to\infty}{\overset{L_2}{\longrightarrow}}\mathbb{E}(Z_0) \text{ as well as } \frac1n \sum_{k=1}^n X_{Y_{-k}}\underset{n\to\infty}{\overset{L_2}{\longrightarrow}}\mathbb{E}(Z_0).
    \end{align*}
\end{theorem}
\begin{proof}
    By Remark 2.7 as well as the proof of Proposition 2.8 d) in \cite{BM-22} and the fact that the assumption $\sup\limits_{t\in\mathbb{R}}\mathbb{E}(X_t^2)<\infty$ implies uniform integrability of $X$, we see that $t\mapsto\mathbb{E}X_t$ is almost periodic. Furthermore, it is clear that 
    \begin{equation}\label{eq_thrm_lln_sampled}
        \sup_{k\in\mathbb{N}}\mathbb{E} X_{Y_k}^2\leq\sup_{t\in\mathbb{R}} \int\limits_{[0,\infty)}\mathbb{E}X_{s}^2\,P_{Y_t}(ds)\leq \sup_{s\in \mathbb{R}} \mathbb{E}X_s^2.
    \end{equation}
    It holds that
    \begin{align*}
       \mathbb{E} \left|\frac{1}{n}\sum\limits_{k=1}^nX_{Y_k}-M(f_X) \right|^2\leq 2 \mathbb{E} \left|\frac{1}{n}\sum\limits_{k=1}^n(X_{Y_k}-f_X(Y_k))\right|^2+ 2\mathbb{E}\left|\frac{1}{n}\sum\limits_{k=1}^nf_X(Y_k) -M(f_X) \right|^2.       
    \end{align*}
    Using Theorem \ref{thrm_robbins} we get
    \begin{align*}
        \mathbb{E}\left|\frac{1}{n}\sum\limits_{k=1}^nf_X(Y_k) -M(f_X) \right|^2\underset{n\to\infty}{\longrightarrow}0.
    \end{align*}
    Define $\xi_t=X_{t}-\mathbb{E}X_{t}=X_{t}-f_X(t)$, then
    \begin{align*}
        \mathbb{E} \left|\frac{1}{n}\sum\limits_{k=1}^n\xi_{Y_k}\right|^2=\frac{1}{n^2}\sum\limits_{k,l=1}^n \mathbb{E} \xi_{Y_k}\xi_{Y_l} =\frac{2}{n^2}\sum\limits_{k=1}^n\sum\limits_{l=1}^{k-1} \mathbb{E} \xi_{Y_k}\xi_{Y_l} + \frac{1}{n^2}\sum\limits_{k=1}^n \mathbb{E} \xi_{Y_k}^2.
    \end{align*}
    Using (\ref{eq_thrm_lln_sampled}), we see that the last sum goes to $0$ as $n\to\infty$, so we only need to look at the first sum on the right hand side. It is clear that
    \begin{align*}
        |\mathbb{E}(\xi_{t}\xi_{t+h})|\leq a_X(h).
    \end{align*}
    Thus we estimate for $k>l$, with $\tilde{Y}=(\tilde{Y}_t)_{t\in\mathbb{R}}$ being an independent copy of $Y$,
    \begin{align*}
        &|\mathbb{E}(\xi_{Y_k}\xi_{Y_l})|=|\mathbb{E}(\xi_{Y_l+\tilde{Y}_{k-l}}\xi_{Y_l})|\\
        \leq&\int_{[0,\infty)^2}|\mathbb{E}(\xi_{t+h}\xi_{t})|P_{\tilde{Y}_{k-l}}(dh)P_{Y_l}(dt) \leq\int_{[0,\infty)^2}a_X(h)P_{\tilde{Y}_{k-l}}(dh)P_{Y_l}(dt)=\mathbb{E}(a_X(\tilde{Y}_{k-l})).
    \end{align*}
    Hence
    \begin{align*}
        \left|\frac{2}{n^2}\sum\limits_{k=1}^n\sum\limits_{l=1}^{k-1} \mathbb{E} \xi_{Y_k}\xi_{Y_l}\right|\leq& \frac{2}{n^2} \sum_{j=1}^{n-1}(n-j)\mathbb{E}(a_X(\tilde{Y}_{j}))\leq\frac{2}{n}\sum_{j=1}^{n-1}\mathbb{E}(a_X(\tilde{Y}_{j})).
    \end{align*}
    Let $\varepsilon>0$ and choose $R>0$ such that $a_X(r)<\varepsilon$ for all $r> R$, then
    \begin{align*}
        \mathbb{E}(a_X(\tilde{Y}_{k}))=& \mathbb{E}(a_X(\tilde{Y}_{k})\One_{\tilde{Y}_{k}>R})+\mathbb{E}(a_X(\tilde{Y}_{k})\One_{\tilde{Y}_{k}\leq R})\\
        \leq& \varepsilon+\sup\limits_{t\in\mathbb{R}}\mathbb{E}(X_t^2) P(\tilde{Y}_{k}\leq R),
    \end{align*}
    where we used that $|\rm{Cov}(X_t,X_{t+h})|\leq \sqrt{\text{Var}(X_t)}\sqrt{\text{Var}(X_{t+h})}\leq \sup\limits_{t\in\mathbb{R}}\mathbb{E}(X_t^2)$. But $\lim\limits_{k\to\infty}P(\tilde{Y}_{k}\leq R)=0$, hence $\mathbb{E}(a_X(\tilde{Y}_{k}))\underset{k\to\infty}{\longrightarrow}0$. Thus by the Césaro Lemma we obtain
    \begin{align*}
        \frac{1}{n}\sum_{j=1}^n \mathbb{E}(a_X(\tilde{Y}_{j}))\underset{n\to\infty}{\longrightarrow}0,
    \end{align*}
    and therefore $\frac1n \sum_{k=1}^n X_{Y_k}\underset{n\to\infty}{\overset{L_2}{\longrightarrow}}M(f_X).$\\
    It remains to show that $M(f_X)=\mathbb{E}(Z_0)$. As was already mentioned, the assumption $\sup\limits_{t\in\mathbb{R}}\mathbb{E}(X_t^2)<\infty$, implies that $X$ is $L_1$-uniformly integrable and hence, Remark \ref{remark_after_stat_thrm}~(iii) shows that
    \begin{align*}
        \mathbb{E}(X_{Y_k})\underset{k\to\infty}{\longrightarrow}\mathbb{E}(Z_0),
    \end{align*}
    Since $X$ and $Y$ are independent, we obtain 
    \begin{align*}
        \mathbb{E}(X_{Y_k})=\int\limits_{[0,\infty)}\mathbb{E}(X_t)\, P_{Y_k}(dt)=\int\limits_{[0,\infty)}f_X(t)\, P_{Y_k}(dt)=\mathbb{E}(f_X(Y_k)).
    \end{align*}
   However with Lemma \ref{lemma1} we see that
    \begin{align*}
        \mathbb{E}(f_X(Y_k))\underset{k\to\infty}{\longrightarrow}M(f_X).
    \end{align*}
    So putting everything together we get $M(f_X)=\mathbb{E}(Z_0)$.\\
    The negative case follows by applying the positive case to the process $\widetilde{X}$ defined by $\widetilde{X}_t:=X_{-t}$. It is clear that if $X$ satisfies the assumptions of Theorem \ref{LLN_sampled}, so does $\widetilde{X}$. Indeed $a_{\widetilde{X}}(h)=\sup\limits_{t\in\mathbb{R}}|\rm{Cov}(\widetilde{X}_t,\widetilde{X}_{t+h})|=\sup\limits_{t\in\mathbb{R}}|\text{Cov}(X_{-t},X_{-t-h})|=\sup\limits_{s\in\mathbb{R}}|\rm{Cov}(X_{s+h},X_s)|=a_X(h)$. Furthermore $X_{Y_{-k}}\overset{a.s.}{=}X_{-Y'_{k}}=\widetilde{X}_{Y_k'}$. Since $f_{\widetilde{X}}(t)=f_X(-t)$, we obtain $M(f_{\widetilde{X}})=M(f_X)$ and thus the result follows also for the negative case.
\end{proof}

\begin{remark}
    A continuous time analog of Theorem \ref{LLN_sampled} can also be proven. The statement in the continuous case is the following:\\
    Under the assumptions of Theorem \ref{LLN_sampled}, we obtain
    \begin{align*}
        \frac1T\int_0^T X_{Y_{t}}\,dt \underset{T\to\infty}{\overset{L_2}{\longrightarrow}} M(f_X)=\mathbb{E}(Z_0), \text{ as well as } \frac1T\int_0^T X_{Y_{-t}}\,dt \underset{T\to\infty}{\overset{L_2}{\longrightarrow}} M(f_X)=\mathbb{E}(Z_0).
    \end{align*}
    The method is basically the same, if we interchange summation with integration and use an integral version of the Césaro Lemma. 
\end{remark}
\noindent
Before we come to the proof of the central limit theorem for the sampled process $X_Y$, we note the following.\\
Let $(Y_t)_{t\in\mathbb{Z}}$ be a non-trivial random sampling and $m\in\mathbb{N}_0$, then for $t\in\mathbb{N}, s>0$ 
\begin{align}\label{exp_decay_coeff}
    P(Y_t\leq m)=P(e^{-Y_t s}\geq e^{-ms})\leq e^{ms}\mathbb{E}(e^{-Y_t s})=e^{ms}(\mathbb{E}(e^{-s\psi_1}))^t=e^{ms-\Phi(s)t},
\end{align}
with $\Phi$ being defined as $\Phi(s):=-\log(\mathbb{E}(e^{-s\psi_1}))$. Since the increment $\psi_1$ is positive, we get $\Phi(s)>0$, in particular $P(Y_t\leq m)$ decays exponentially as $t\to\infty$. \\
\begin{definition}
Let $m\geq 0$. A stochastic process
$X=(X_t)_{t\in\mathbb R}$ is called \emph{$m$-dependent} if, for every
$s\in\mathbb R$, the $\sigma$-algebras
\begin{align*}
    \sigma(X_t:t\leq s)
    \qquad\text{and}\qquad
    \sigma(X_t:t>s+m)    
\end{align*}
are independent.
\end{definition}
\noindent
We can now formulate a central limit theorem for randomly sampled $m$-dependent processes that are almost periodic in distributional sense.
\begin{theorem}\label{CLT_sampled}
    Let $(X_t)_{t\in\mathbb{R}}$ be a stochastic process, which is almost periodic in distributional sense, $m$-dependent for some $m\in\mathbb{N}$ and $L_2$-uniformly integrable. Let $(Y_g)_{g\in\mathbb{Z}}$ be a random sampling such that $|\widehat{\mathcal{L}}(Y_1)(z)|\ne 1$ for all $z\ne 0$. Assume further that $Y$ and $X$ are independent. Define 
    \begin{align*}
        f_X(t):=\mathbb{E}(X_t), \qquad \xi_t:=X_t-f_X(t),
    \end{align*}
    and let $Z$ be the stationary limit process from Theorem \ref{Theorem_stat_Z}, applied to the almost periodic in distributional sense process $(\xi_t)_{t\in\mathbb{R}}$.
    Then
    \begin{align}\label{abs_conv_acf}
        \sum_{h=1}^\infty|\gamma_Z(h)|<\infty \quad\text{ and }\quad \sigma^2:=\gamma_Z(0)+2\sum_{h=1}^\infty \gamma_Z(h)=\sum_{h\in\mathbb{Z}}\gamma_Z(h)\geq 0,
    \end{align}
    where $\gamma_Z$ denotes the autocovariance function of $Z$. Furthermore:
    \begin{enumerate}
        \item If $\sigma^2>0$, then
        \begin{align*}
            \frac{\sum_{k=1}^n (X_{Y_k}-f_X(Y_k))}{\sqrt{n}\sigma}\underset{n\to\infty}{\overset{d}{\longrightarrow}}N(0,1) \text{ as well as }\frac{\sum_{k=1}^n (X_{Y_{-k}}-f_X(Y_{-k}))}{\sqrt{n}\sigma}\underset{n\to\infty}{\overset{d}{\longrightarrow}}N(0,1).
        \end{align*}
        \item If $\sigma^2=0$, then
        \begin{align*}
            \frac{\sum_{k=1}^n (X_{Y_k}-f_X(Y_k))}{\sqrt{n}}\underset{n\to\infty}{\overset{L_2}{\longrightarrow}}0 \text{ as well as }\frac{\sum_{k=1}^n (X_{Y_{-k}}-f_X(Y_{-k}))}{\sqrt{n}}\underset{n\to\infty}{\overset{L_2}{\longrightarrow}}0.
        \end{align*}
    \end{enumerate}
\end{theorem}
\noindent
We will prove this theorem with a series of lemmata. First we establish the notion of an $L_2$-mixingale, which can be found in \cite[Definition 2.1]{McLeish1975}.
\begin{definition}[McLeish \cite{McLeish1975}]\label{def_mixingale}
    Let $(\mathcal{F}_j)_{j\in\mathbb Z}$ be a nondecreasing sequence of $\sigma$-algebras. A centered square-integrable sequence $(X_i)_{i\in\mathbb{N}}$ is an \emph{$L_2$-mixingale} if there is a positive nonincreasing sequence $\theta_r\downarrow0$ such that
    \begin{align}\label{mixingale_equation_1}
        ||{\mathbb{E}(X_i\mid\cF_{i-r})}||_{L_2}\leq\theta_r       
    \end{align}
    and
    \begin{align}\label{mixingale_equation_2}
        ||{X_i-\mathbb{E}(X_i|\mathcal{F}_{i+r})}||_{L_2}\leq\theta_{r+1}
    \end{align}
    for all $i\in \mathbb{N}$ and $r\in\mathbb{N}_0$.\\
    The sequence $(\theta_r)_{r\in\mathbb{N}_0}$ is said to be of \emph{size} $-\frac{1}{2}$, if there exists a sequence $(L_r)_{r\in\mathbb{N}_0}$ of positive numbers such that $\sum_{r=1}^\infty \frac{1}{rL_r}<\infty$, $L_r-L_{r-1}=O\left(\frac{L_r}{r}\right)$ as $r\to\infty$, $L_r$ is eventually non-decreasing and $\theta_r=o\left(\frac{1}{r^{1/2}L_r}\right)$ as $r\to\infty$. 
\end{definition}

\begin{proposition}[Consequence of McLeish's Theorem 2.6]\label{prop_McLeish}
    Let $((X_i)_{i\in\mathbb{N}},(\mathcal{F}_j)_{j\in\mathbb{Z}})$ be a centered $L_2$-mixingale whose coefficients $(\theta_r)_{r\in\mathbb{N}_0}$ have size $-1/2$. Define $S_n=\sum_{i=1}^n X_i$ and assume that
    \begin{enumerate}
      \item $(X_i)_{i\in\mathbb{N}}$ is $L_2$-uniformly integrable,
      \item $\frac{1}{n}\mathbb{E} S_n^2\underset{n\to\infty}{\longrightarrow}\sigma^2>0$,
      \item
      \begin{align*}
            \mathbb{E}\left(
           \frac{(S_{k+n}-S_k)^2}{n}
           \Big| \cF_{k-r}
         \right)
         \overset{L_1}{\longrightarrow}\sigma^2
      \end{align*}
      as $\min\{r,k,n\}\to\infty$.
    \end{enumerate}
    Then
    \begin{align*}
        \frac{S_n}{\sqrt n}\underset{n\to\infty}{\overset{d}{\longrightarrow}} N(0,\sigma^2).
    \end{align*}
\end{proposition}
\begin{proof}
    Actually McLeish \cite[Theorem 2.6]{McLeish1975} obtained a functional central limit theorem and Proposition \ref{prop_McLeish} is an immediate consequence of that.
\end{proof}
\noindent
Throughout the remainder of the section we assume $\eta=(\eta_t)_{t\in\mathbb{R}}$ to be jointly measurable, centered, almost periodic in distributional sense, $m$-dependent and bounded by $|\eta_t|\leq K$ almost surely for all $t\in\mathbb{R}$, for some $K>0$. Let $Y=(Y_k)_{k\in\mathbb{Z}}$ be a random sampling as in Theorem \ref{CLT_sampled}, which is independent of $\eta$ and which satisfies $|\widehat{\mathcal{L}}(Y_1)(z)|\ne 1$ for all $z\ne 0$. By $Z^\eta$ we denote the stationary limit process from Theorem \ref{Theorem_stat_Z}. Furthermore we define
\begin{align*}
    \mathcal{Y}:=\sigma(Y_k: k\in\mathbb{N}_0)
\end{align*}
as well as for $u\geq 0$
\begin{align*}
    c_u(t):=\mathbb{E}(\eta_t \eta_{t+u}), \quad t\in\mathbb{R}.
\end{align*}
We collect the following preliminaries about the function $c_u$ and the autocovariance function $\gamma_\eta$ of $Z^\eta$.
\begin{lemma}\label{prop_c_u}
    \begin{enumerate}
        \item The function $(t,u)\mapsto c_u(t)$ is measurable,
        \item for $u\geq 0$, $c_u\in AP(\mathbb{R},\mathbb{R})$ and $||c_u||_\infty\leq K^2$,
        \item $c_u(t)=0$ $\forall t\in\mathbb{R}$ if $u>m$,
        \item for $h\in\mathbb{N}_0$, $\gamma_\eta(h):={\rm{Cov}}(Z_0^\eta, Z_h^\eta)=\mathbb{E}(Z^\eta_0Z^\eta_h)=\int_{[0,\infty)}M(c_u)P_{Y_h}(du)$,
        \item for $h\in\mathbb{N}_0$, $|\gamma_\eta(h)|\leq K^2 P(Y_h\leq m)$,
        \item $\sigma_\eta^2:=\gamma_\eta(0)+2\sum_{h=1}^\infty\gamma_\eta(h)\in [0,\infty)$ with absolute convergence of the series.
    \end{enumerate}
\end{lemma}

\begin{proof}
    The process $\eta$ is $L_2$-uniformly integrable, as it is a.s. bounded. Proposition 2.8(d) in \cite{BM-22} therefore gives $c_u\in AP(\mathbb{R},\mathbb{R})$ for every fixed $u\geq 0$. Furthermore joint measurability of $(t,u)\mapsto c_u(t)$ follows from joint measurability of $\eta$ and Fubini's theorem, exactly as in the proof of Theorem 7.2 in \cite{BM-22}. Also since for every $N\in\mathbb{N}$ the function $u\mapsto \frac{1}{2N}\int_{(-N,N)}c_u(t)\,dt$ is measurable and converges pointwise to $M(c_u)$, we get that $u\mapsto M(c_u)$ is measurable. Since $\eta$ is bounded, $||c_u||_\infty\leq K^2$ is immediate. Furthermore $m$-dependence and centering of $\eta$ immediately give $c_u(t)=0$ $\forall t\in\mathbb{R}$, whenever $u>m$, so (i)-(iii) are proven.\\
    Independence of $\eta$ and $Y$ together with the stationary and independent increments property of $Y$ gives, for $r,h\in\mathbb{N}_0$
    \begin{align*}
        \mathbb{E}(\eta_{Y_r}\eta_{Y_{r+h}})=\mathbb{E}(c_{Y_{r+h}-Y_r}(Y_r))=\int_{[0,\infty)}\mathbb{E}(c_u(Y_r))\,P_{Y_h}(du).
    \end{align*}
    With Lemma \ref{lemma1} we get for each $u\geq 0$ that $\mathbb{E}(c_u(Y_r))\underset{r\to\infty}{\longrightarrow}M(c_u)$. Lebesgue's dominated convergence theorem, which is applicable because of (ii), therefore gives
    \begin{align}\label{eq_prop_c_u_1}
        \mathbb{E}(\eta_{Y_r}\eta_{Y_{r+h}})\underset{r\to\infty}{\longrightarrow}\int_{[0,\infty)}M(c_u)P_{Y_h}(du).
    \end{align}
    Since $Y$ is independent of $\eta$, which is centered, a conditioning argument shows that $\mathbb{E}(\eta_{Y_r})=0$ for all $r\in\mathbb{N}_0$. By Theorem \ref{Theorem_stat_Z}, we get $(\eta_{Y_r},\eta_{Y_{r+h}})\underset{r\to\infty}{\overset{d}{\longrightarrow}}(Z_0^\eta, Z_h^\eta)$ for all $h\in\mathbb{N}_0$. Since $\eta$ is uniformly bounded Remark \ref{remark_after_stat_thrm}(iii) ensures convergence of the first moments and the product moments to the corresponding moments of $Z^\eta$. Comparing with (\ref{eq_prop_c_u_1}) yields (iv).\\
    Using (iv) we can estimate
    \begin{align*}
        |\gamma_\eta(h)|\leq\int_{[0,\infty)}|M(c_u)|\,P_{Y_h}(du)=\int_{[0,m]}|M(c_u)|P_{Y_h}(du)\leq K^2 P(Y_h\leq m),
    \end{align*}
    which proves (v). Absolute convergence of the series in (vi) now follows from (v) and because $P(Y_h\leq m)$ decays exponentially. Stationarity of $Z^\eta$ yields
    \begin{align*}
        0\leq \frac{1}{n}\text{Var}\left(\sum_{j=1}^n Z_j^\eta\right)=\gamma_\eta(0)+2\sum_{i=1}^{n-1}\left(1-\frac{h}{n}\right)\gamma_\eta(h).
    \end{align*}
    Since we have already established absolute summability we are able to conclude $\sigma_\eta^2\geq 0$, by letting $n\to\infty$.
\end{proof}

\begin{lemma}\label{lemma_moving_block_var_cond}
    Let $\sigma_\eta^2$ be defined as in Lemma \ref{prop_c_u}, then
    \begin{align*}
        \sup\limits_{k\geq 0}\left|\left|\frac{1}{n}\mathbb{E}\left(\left(\sum_{j=k+1}^{k+n}\eta_{Y_j}\right)^2\Bigg |\mathcal Y\right)-\sigma_\eta^2\right|\right|_{L_1}\underset{n\to\infty}{\longrightarrow}0.
    \end{align*}
\end{lemma}
\begin{proof}
    Using the independence of $\eta$ and $Y$, we see
    \begin{align*}
        \mathbb{E}(\eta_{Y_i}\eta_{Y_{i+h}}|\mathcal{Y})=c_{Y_{i+h}-Y_i}(Y_i)
    \end{align*}
    for $i,\in\mathbb{N}, h\in\mathbb{N}_0$. Let $n\in\mathbb{N}$ and $k\in\mathbb{N}_0$. We expand
    \begin{align*}
        \left(\sum_{j=k+1}^{k+n}\eta_{Y_j}\right)^2=\sum_{i=k+1}^{k+n}\eta_{Y_i}^2+2\sum_{i={k+1}}^{k+n-1}\sum_{j=i+1}^{k+n}\eta_{Y_i}\eta_{Y_j}=\sum_{i=k+1}^{k+n}\eta_{Y_i}^2+2\sum_{h={1}}^{n-1}\sum_{i=k+1}^{k+n-h}\eta_{Y_i}\eta_{Y_{i+h}}.
    \end{align*}
    Conditionally on $\mathcal{Y}$ we therefore obtain
    \begin{align}\label{eq_lemma_cond_mov_block_6}
        \frac{1}{n}\mathbb{E}\left(\left(\sum_{j=k+1}^{k+n}\eta_{Y_j}\right)^2\Bigg|\mathcal{Y}\right)=\frac{1}{n}\sum_{i=k+1}^{k+n}c_0(Y_i)+\frac{2}{n}\sum_{h={1}}^{n-1}\sum_{i=k+1}^{k+n-h}c_{Y_{i+h}-Y_i}(Y_i).
    \end{align}
    Fix $h\in\mathbb{N}_0$ and let $n>h$. Our next goal is to show
    \begin{align}\label{eq_lemma_cond_mov_block_1}
        \sup\limits_{k\geq 0}\left|\left|\frac{1}{n}\sum_{i=k+1}^{k+n-h}c_{Y_{i+h}-Y_i}(Y_i)-\gamma_\eta(h)\right|\right|_{L_2}\underset{n\to\infty}{\longrightarrow}0.
    \end{align}
    by showing
    \begin{align}\label{eq_lemma_cond_mov_block_2}
        \sup\limits_{k\geq 0}\left|\left|\frac{1}{n}\sum_{i=k+1}^{k+n-h}\big(c_{Y_{i+h}-Y_i}(Y_i)-g_h(Y_i)\big)\right|\right|_{L_2}\underset{n\to\infty}{\longrightarrow}0
    \end{align}
    and 
    \begin{align}\label{eq_lemma_cond_mov_block_4}
        \sup\limits_{k\geq 0}\left|\left|\frac{1}{n}\sum_{i=k+1}^{k+n-h}\big(g_h(Y_i)-\gamma_\eta(h)\big)\right|\right|_{L_2}\underset{n\to\infty}{\longrightarrow}0.
    \end{align}
    with $g_h$ being defined as
    \begin{align*}
        g_h(t):=\int_{[0,\infty)}c_u(t)\, P_{Y_h}(du)=\mathbb{E}(c_{Y_h}(t))=\mathbb{E}(c_{Y_{i+h}-Y_i}(t)) \qquad \forall i\in\mathbb{N}
    \end{align*}
    for $t\in\mathbb{R}$. We start with (\ref{eq_lemma_cond_mov_block_2}). Lemma \ref{prop_c_u} gives joint measurability and boundedness of $c_u(t)$, which implies that $g_h$ is measurable and $|g_h(t)|\leq K^2$. Define further 
    \begin{align*}
        \mathcal{G}_i:=\sigma(\psi_1,...,\psi_i)=\sigma(Y_1,...,Y_i).
    \end{align*}
    Since $Y_{i+h}-Y_i$ is independent of $\mathcal{G}_i$ we obtain $\mathbb{E}(c_{Y_{i+h}-Y_i}(Y_i)|\mathcal{G}_i)=g_h(Y_i)$, hence 
    \begin{align*}
        \mathbb{E}(c_{Y_{i+h}-Y_i}(Y_i)-g_h(Y_i)|\mathcal{G}_i)=0 \qquad \forall i\in\mathbb{N}, h\in\mathbb{N}_0.
    \end{align*} 
    Now for $h\in\mathbb{N}_0$ and $i\leq j$ with $j-i\geq h$, $c_{Y_{i+h}-Y_i}(Y_i)-g_h(Y_i)$ is $\mathcal{G}_i$-measurable, hence using the tower property we get
    \begin{align*}
        &\mathbb{E}\big[\big((c_{Y_{i+h}-Y_i}(Y_i)-g_h(Y_i)\big)\big((c_{Y_{j+h}-Y_j}(Y_j)-g_h(Y_j)\big)\big]\\
        &=\mathbb{E}\big[\big((c_{Y_{i+h}-Y_i}(Y_i)-g_h(Y_i)\big)\mathbb{E}\big((c_{Y_{j+h}-Y_j}(Y_j)-g_h(Y_j)\big|\mathcal{G}_j\big)\big]=0.
    \end{align*}
    So $\mathbb{E}\big[\big((c_{Y_{i+h}-Y_i}(Y_i)-g_h(Y_i)\big)\big((c_{Y_{j+h}-Y_j}(Y_j)-g_h(Y_j)\big)\big]=0$ whenever $|i-j|\geq h$.
    We expand
    \begin{align*}
        &\mathbb{E}\left(\frac{1}{n}\sum_{i=k+1}^{k+n-h}(c_{Y_{i+h}-Y_i}(Y_i)-g_h(Y_i))\right)^2\\
        &=\frac{1}{n^2}\sum_{i=k+1}^{k+n-h}\sum_{j=k+1}^{k+n-h}\mathbb{E}\big[\big(c_{Y_{i+h}-Y_i}(Y_i)-g_h(Y_i)\big)\big(c_{Y_{j+h}-Y_j}(Y_j)-g_h(Y_j)\big)\big].
    \end{align*}
    Only summands for which $|i-j|< h$ are potentially not equal to $0$. For each $i$ there are at most $2h$ choices of such $j$ ($2h-1$ choices if $h\ne 0$) and since there are $n-h$ choices of $i$ the sum contains at most $(n-h)2h$ summands which are potentially not equal to $0$. Since also $|(c_{Y_{i+h}-Y_i}(Y_i)-g_h(Y_i)|\leq 2K^2$, each summand is bounded by $4K^4$. Thus
    \begin{align*}
        &\mathbb{E}\left(\frac{1}{n}\sum_{i=k+1}^{k+n-h}((c_{Y_{i+h}-Y_i}(Y_i)-g_h(Y_i))\right)^2\leq \frac{4K^4(n-h)2h}{n^2}.
    \end{align*}
    For the case $h=0$, we have $Y_{i+h}-Y_i=0$, i.e. $g_0(t)=c_0(t)$ and thus the left-hand side of the previous inequality is $0$. Hence for each fixed $h\in\mathbb{N}_0$ (\ref{eq_lemma_cond_mov_block_2}) follows. \\
    Next we want to prove (\ref{eq_lemma_cond_mov_block_4}). For $r\in\mathbb{N}$ and $h\in\mathbb{N}_0$ we define
    \begin{align*}
        \varepsilon_h(r):=\int_{[0,\infty)}\sup\limits_{x\in\mathbb{R}}|\mathbb{E}(c_u(x+Y_r))-M(c_u))|\, P_{Y_h}(du).
    \end{align*}
    This is well-defined because $u\mapsto M(c_u)$ is measurable as was shown in the proof of Lemma \ref{prop_c_u}. Furthermore for each $u,r$ fixed the function $x\mapsto \mathbb{E}(c_u(x+Y_r))$ is continuous, which follows from continuity and boundedness of $c_u\in AP(\mathbb{R},\mathbb{R})$ and Lebesgue's dominated convergence theorem. So the supremum in the integral of the definition of $\varepsilon_h$ could also be taken over $\mathbb{Q}$ instead of $\mathbb{R}$. Since for each $q\in\mathbb{Q}$ the function $u\mapsto\mathbb{E}(c_u(q+Y_r))$ is measurable because of joint measurability of $c_u$ and since countable suprema of measurable functions of $u$ are measurable, the integrand in $\varepsilon_h$ is measurable and $\varepsilon_h$ therefore well-defined.\\
    Now since $c_u$ is almost periodic for each fixed $u$, Corollary \ref{lemma_uniformity} shows that
    \begin{align*}
        \sup\limits_{x\in\mathbb{R}}|\mathbb{E}(c_u(x+Y_r)-M(c_u))|\underset{r\to\infty}{\longrightarrow}0.
    \end{align*}
    Since $\sup\limits_{x\in\mathbb{R}}|\mathbb{E}(c_u(x+Y_r)-M(c_u))|\leq 2K^2$, Lebesgue's dominated convergence theorem again gives
    \begin{align*}
        \varepsilon_h(r)\underset{r\to\infty}{\longrightarrow}0
    \end{align*}
    for each fixed $h\in\mathbb{N}_0$. Using Lemma \ref{prop_c_u} (iv) and Fubini's theorem we obtain for $x\in\mathbb{R}$
    \begin{align*}
        \mathbb{E}(g_h(x+Y_r))-\gamma_\eta(h)=\int_{[0,\infty)}\mathbb{E}(c_u(x+Y_r))-M(c_u)\,P_{Y_h}(du),
    \end{align*}
    which implies
    \begin{align}\label{eq_lemma_cond_mov_block_3}
        \sup\limits_{x\in\mathbb{R}}|\mathbb{E}(g_h(x+Y_r))-\gamma_\eta(h)|\leq \varepsilon_h(r).
    \end{align}
    Let $j>i$ then again $Y_j-Y_i$ is independent of $\mathcal{G}_i$ and has law $P_{Y_{j-i}}$, hence 
    \begin{align*}
        &\mathbb{E}(g_h(Y_j)-\gamma_\eta(h)|\mathcal{G}_i)=\mathbb{E}(g_h(Y_i+(Y_j-Y_i))|\mathcal{G}_i)-\gamma_\eta(h)\\
        &=\int\limits_{[0,\infty)}g_h(Y_i+u)\,P_{Y_{j-i}}(du)-\gamma_\eta(h).
    \end{align*}
    We can now estimate 
    \begin{align*}
        &|\mathbb{E}(g_h(Y_j)-\gamma_\eta(h)|\mathcal{G}_i)|=\left|\int_{[0,\infty)}g_h(Y_i+u)\,P_{Y_{j-i}}(du)-\gamma_\eta(h)\right|\\
        &\leq \sup\limits_{x\in\mathbb{R}}\left|\int_{[0,\infty)}g_h(x+u)\,P_{Y_{j-i}}(du)-\gamma_\eta(h)\right|\\
        &=\sup\limits_{x\in\mathbb{R}}|\mathbb{E}(g_h(x+Y_{j-i}))-\gamma_\eta(h)|\underset{(\ref{eq_lemma_cond_mov_block_3})}{\leq} \varepsilon_h({j-i}).
    \end{align*}
    Since $|g_h(Y_i)-\gamma_\eta(h)|\leq 2K^2$ we get with the tower property
    \begin{align*}
        &|\mathbb{E}\big[\big(g_h(Y_i)-\gamma_\eta(h)\big)\big(g_h(Y_j)-\gamma_\eta(h)\big)\big]|=|\mathbb{E}\big[\big(g_h(Y_i)-\gamma_\eta(h)\big)\mathbb{E}\big(g_h(Y_j)-\gamma_\eta(h)|\mathcal{G}_i\big)\big]|\\
        &\leq 2K^2 \varepsilon_h(j-i).
    \end{align*}
    We now expand $\left(\frac{1}{n}\sum_{i=k+1}^{k+n-h}\left(g_h(Y_i)-\gamma_\eta(h)\right)\right)^2$. The sum inside the square has $n-h$ terms. For $1\leq r\leq n-h-1$ there are exactly $n-h-r$ pairs of indices $i<j$ with $j-i=r$. The diagonal terms can be estimated by $4K^4$, the non-diagonal terms by $2K^2\varepsilon_h(r)$. Thus
    \begin{align*}
        &\mathbb{E}\left(\frac{1}{n}\sum_{i=k+1}^{k+n-h}\left(g_h(Y_i)-\gamma_\eta(h)\right)\right)^2\leq \frac{4K^4(n-h)}{n^2}+\frac{4K^2}{n^2}\sum_{r=1}^{n-h-1}(n-h-r)\varepsilon_h(r)\\
        &\leq \frac{4K^4}{n}+\frac{4K^2}{n}\sum_{r=1}^{n-1}\varepsilon_h(r),
    \end{align*}
    where the second term converges to $0$ as $n\to\infty$ by the Lemma of Cesàro. Since the bound in the previous inequality does not depend on $k$, we get (\ref{eq_lemma_cond_mov_block_4}). This now establishes (\ref{eq_lemma_cond_mov_block_1}), since
    \begin{align*}
        &\sup\limits_{k\geq 0}\left|\left|\frac{1}{n}\sum_{i=k+1}^{k+n-h}c_{Y_{i+h}-Y_i}(Y_i)-\gamma_\eta(h)\right|\right|_{L_2}\\
        &\leq
        \sup\limits_{k\geq 0}\left|\left|\frac{1}{n}\sum_{i=k+1}^{k+n-h}\big(c_{Y_{i+h}-Y_i}(Y_i)-g_h(Y_i)\big)\right|\right|_{L_2}\\
        &+\sup\limits_{k\geq 0}\left|\left|\frac{1}{n}\sum_{i=k+1}^{k+n-h}\big(g_h(Y_i)-\gamma_\eta(h)\big)\right|\right|_{L_2}+\frac{h}{n}|\gamma_\eta(h)|.        
    \end{align*}
    The first term converges to $0$ as $n\to\infty$ by (\ref{eq_lemma_cond_mov_block_2}), the second by (\ref{eq_lemma_cond_mov_block_4}) and the last because $h\in\mathbb{N}_0$ is fixed.\\
    So far we have only treated fixed lags $h\in\mathbb{N}_0$. We now want to establish uniform control of large lags. We know that $c_u(t)=0$, if $u>m$ and that (similar to the proof of Lemma \ref{prop_c_u} (v)) $\mathbb{E}|c_{Y_{i+h}-Y_i}(Y_i)|\leq K^2 P({Y_h\leq m})$. We now fix $H\in\mathbb{N}_ 0$. Then
    \begin{align*}
        &\mathbb{E}\left(\sum_{h=H+1}^{n-1}\left|\frac{1}{n}\sum_{i=k+1}^{k+n-h}c_{Y_{i+h}-Y_i}(Y_i)\right|\right)\leq \frac{1}{n}\sum_{h=H+1}^{n-1}\sum_{i=k+1}^{k+n-h}\mathbb{E}|c_{Y_{i+h}-Y_i}(Y_i)|\\
        &\leq K^2 \sum_{h=H+1}^{n-1}\frac{n-h}{n}P(Y_h\leq m)\leq K^2\sum_{h=H+1}^\infty P(Y_h\leq m).
    \end{align*}
    Since the last term is independent of $k$ and $n$ and since the sum in the last term is absolutely converging, we get
    \begin{align}\label{eq_lemma_cond_mov_block_5}
        \sup\limits_{k\geq 0}\sup\limits_{n\geq 1}\mathbb{E}\left(\sum_{h=H+1}^{n-1}\left|\frac{1}{n}\sum_{i=k+1}^{k+n-h}c_{Y_{i+h}-Y_i}(Y_i)\right|\right)\leq K^2 \sum_{h=H+1}^\infty P(Y_h\leq m) \underset{H\to\infty}{\longrightarrow}0.
    \end{align}
    For fixed $H\in\mathbb{N}_ 0$ and $n>H$ we write
    \begin{align*}
        \sigma_\eta^2=\gamma_\eta(0)+2\sum_{h=1}^H\gamma_\eta(h)+2\sum_{h=H+1}^\infty\gamma_\eta(h).
    \end{align*}
    Using (\ref{eq_lemma_cond_mov_block_6}) this gives the following estimate
    \begin{align*}
        &\sup\limits_{k\geq 0}\left|\left|\frac{1}{n}\mathbb{E}\left(\left(\sum_{j=k+1}^{k+n}\eta_{Y_j}\right)^2\Bigg |\mathcal Y\right)-\sigma_\eta^2\right|\right|_{L_1}\\
        &\leq \sup\limits_{k\geq 0}\left|\left|\frac{1}{n}\sum_{i=k+1}^{k+n}c_{0}(Y_i)-\gamma_\eta(0)\right|\right|_{L_1}\\
        &+ 2\sum_{h=1}^H \sup\limits_{k\geq 0}\left|\left|\frac{1}{n}\sum_{i=k+1}^{k+n-h}c_{Y_{i+h}-Y_i}(Y_i)-\gamma_\eta(h)\right|\right|_{L_1}\\
        &+ 4K^2 \sum_{h=H+1}^\infty P(Y_h\leq m),
    \end{align*}
    where for the last term we used the estimate (\ref{eq_lemma_cond_mov_block_5}) as well as $\sum_{h=H+1}^\infty|\gamma_\eta(h)|\leq K^2\sum_{h=H+1}^\infty P(Y_h\leq m)$, which is due to Lemma \ref{prop_c_u} (v). Since $||\cdot||_{L_1}\leq ||\cdot||_{L_2}$ we get with (\ref{eq_lemma_cond_mov_block_1}) for fixed $H\in\mathbb{N}_0$
    \begin{align*}
        \limsup\limits_{n\to\infty}\sup\limits_{k\geq 0}\left|\left|\frac{1}{n}\mathbb{E}\left(\left(\sum_{j=k+1}^{k+n}\eta_{Y_j}\right)^2\Bigg |\mathcal Y\right)-\sigma_\eta^2\right|\right|_{L_1} \leq 4K^2 \sum_{h=H+1}^\infty P(Y_h\leq m).
    \end{align*}
    Letting now $H\to\infty$ completes the proof.
\end{proof}
\noindent
Define 
\begin{align*}
    \mathcal{F}_j:=\begin{cases}
        \mathcal{Y}, & j\leq 0,\\
        \sigma(\mathcal{Y} \cup \sigma(\eta_{Y_1},...,\eta_{Y_j})), & j\geq 1.
    \end{cases}
\end{align*}
\begin{lemma}\label{lemma_mixingale_trunc}
    $((\eta_{Y_i})_{i\in\mathbb{N}},(\mathcal{F}_j)_{j\in\mathbb{Z}})$ is a $L_2$-uniformly integrable centered $L_2$-mixingale. More precisely, for $r\in\mathbb{N}_0$,
    \begin{align*}
        \sup\limits_{i\in\mathbb{N}}||\mathbb{E}(\eta_{Y_i}|\mathcal{F}_{i-r})||_{L_2}\leq K P(Y_r\leq m)^{1/2},
    \end{align*}
    and $\eta_{Y_i}=\mathbb{E}(\eta_{Y_i}|\mathcal{F}_{i+r}).$ We can furthermore choose exponentially decaying mixingale coefficients, which have size $-1/2$.
\end{lemma}
\begin{proof}
    That $\mathbb{E}(\eta_{Y_i})=0$ for all $i\in\mathbb{N}$ was already observed in the proof of Lemma \ref{prop_c_u}. Also boundedness of $\eta$ immediately implies $L_2$-uniform integrability of $\eta_Y$.\\
    Since for $i\in\mathbb{N}$, $\eta_{Y_i}$ is $\mathcal{F}_i$-measurable and $(\mathcal{F}_j)_{j\in\mathbb{Z}}$ is increasing, we immediately get $\eta_{Y_i}=\mathbb{E}(\eta_{Y_i}|\mathcal{F}_{i+r})$ for $r\in\mathbb{N}_0$ and thus $((\eta_{Y_i})_{i\in\mathbb{N}},(\mathcal{F}_j)_{j\in\mathbb{Z}})$ fulfills (\ref{mixingale_equation_2}) for any positive sequence $\theta_r\downarrow 0$.\\
    For the proof of (\ref{mixingale_equation_1}), if $i-r\leq 0$, then $\mathcal{F}_{i-r}=\mathcal{Y}$ and $\mathbb{E}(\eta_{Y_i}|\mathcal Y)=0$ since $\eta$ is centered and independent of $Y$, we therefore only need to consider $i-r\geq 1$. Define $D:=\{Y_i-Y_{i-r}>m\} \in \mathcal Y\subset\mathcal{F}_{i-r}$. Our next goal is to prove that 
    \begin{align}\label{eq_mixingale_proof_1}
        \mathbb{E}(\One_A \One_D\eta_{Y_i})=0 \qquad \forall A\in \mathcal{F}_{i-r}.
    \end{align}
    Define $\overline{Y}:=(Y_k)_{k\in\mathbb{N}_0}$. Since $\eta$ and $Y$ are independent are independent, we can realise $(\eta,\overline{Y})$ on a product space $(\Omega_\eta, \mathcal{F}_\eta, P_\eta)\otimes(\Omega_{\overline{Y}}, \mathcal{F}_{\overline{Y}}, P_{\overline{Y}})$ so that $\eta_{Y_i}(\omega)=\eta_{Y_i(\omega_Y)}(\omega_\eta)$. Since $A\in\mathcal{F}_{i-r}$, there exists a $B\in \left(\bigotimes\limits_{n\in\mathbb{N}_0} \mathcal{B}([0,\infty))\right)\otimes \mathcal{B}(\mathbb{R}^{i-r})$ such that $\One_A=\One_B \circ (\overline{Y}, \eta_{Y_1},...,\eta_{Y_{i-r}})$. Then we obtain
    \begin{align}\label{eq_mixingale_proof_2}
        \mathbb{E}(\One_A \One_D\eta_{Y_i})=\int\limits_{\Omega_{\overline{Y}}}\One_{\{Y_i(\omega_Y)-Y_{i-r}(\omega_Y)>m\}}G(\omega_Y)\,P_{\overline{Y}}(d\omega_Y),
    \end{align}
    where
    \begin{align*}
        &G(\omega_Y):=\int\limits_{\Omega_\eta}\eta_{Y_i(\omega_Y)}(\omega_\eta)\One_B(\overline{Y}(\omega_Y),\eta_{Y_1(\omega_Y)}(\omega_\eta),...,\eta_{Y_{i-r}(\omega_Y)}(\omega_\eta))\,P_\eta(d\omega_\eta)\\
        &=\mathbb{E}^\eta[\eta_{Y_i(\omega_Y)}\One_B(\overline{Y}(\omega_Y),\eta_{Y_1(\omega_Y)},...,\eta_{Y_{i-r}(\omega_Y)})]
    \end{align*}
    with $\mathbb{E}^\eta$ denoting the expected value with respect to $\eta$. If for fixed $\omega_Y$, $Y_i(\omega_Y)-Y_{i-r}(\omega_Y)>m$, the $m$-dependence property of $\eta$ yields
    \begin{align*}
        G(\omega_Y)=\mathbb{E}^\eta(\eta_{Y_i(\omega_Y)})\mathbb{E}^\eta(\One_B(\overline{Y}(\omega_Y),\eta_{Y_1(\omega_Y)},...,\eta_{Y_{i-r}(\omega_Y)}))=0,
    \end{align*}
    since $\mathbb{E}^\eta(\eta_{Y_i(\omega_Y)})=0$ as $\eta$ is centered. Hence the integrand in (\ref{eq_mixingale_proof_2}) is $P_\eta \otimes P_{\overline{Y}}$-almost surely equal to $0$, thus establishing (\ref{eq_mixingale_proof_1}). Now (\ref{eq_mixingale_proof_1}) implies $\One_D\mathbb{E}(\eta_{Y_i}|\mathcal{F}_{i-r})=\mathbb{E}(\One_D\eta_{Y_i}|\mathcal{F}_{i-r})=0$ for all $i\in\mathbb{N}$ and $r\in\mathbb{N}_0$. 
    We write - using the conditional Jensen inequality -
    \begin{align*}
        &||\mathbb{E}(\eta_{Y_i}|\mathcal{F}_{i-r})||_{L_2}^2=\mathbb{E}[\One_{D^c}|\mathbb{E}(\eta_{Y_i}|\mathcal{F}_{i-r})|^2]\leq \mathbb{E}[\One_{D^c}\mathbb{E}(\eta_{Y_i}^2|\mathcal{F}_{i-r})]\\
        &\leq K^2 P(D^c)= K^2P(Y_r\leq m).
    \end{align*}
    Since the right-hand side does not depend on $i$, we get
    \begin{align*}
        \sup\limits_{i\in\mathbb{N}}||\mathbb{E}(\eta_{Y_i}|\mathcal{F}_{i-r})||_{L_2}\leq K P(Y_r\leq m)^{1/2}.
    \end{align*}
    We now choose coefficients $(\theta_r)_{r\in\mathbb{N}_0}$ such that $((\eta_{Y_i})_{i\in\mathbb{N}},(\mathcal{F}_j)_{j\in\mathbb{Z}})$ is an $L_2$-mixingale, with coefficients of size $-1/2$. We use the estimate (\ref{exp_decay_coeff}) to get for $r\in\mathbb{N}_0$ and any fixed $s>0$
    \begin{align*}
        KP(Y_r\leq m)^{1/2}\leq K e^{ms/2}e^{-\Phi(s)r/2}=:\theta_r.
    \end{align*}
    It is now easy to verify that, using $L(r)=r$, the coefficients $(\theta_r)_{r\in\mathbb{N}_0}$ satisfy McLeish's size $-1/2$ condition and that $((\eta_{Y_i})_{i\in\mathbb{N}},(\mathcal{F}_j)_{j\in\mathbb{Z}})$ satisfies (\ref{mixingale_equation_1}). This completes the proof.
\end{proof}

\begin{lemma}\label{lemma_moving_block_cond_full}
    Let $\sigma_\eta^2$ be defined as in Lemma \ref{prop_c_u}, then for $r,n\in\mathbb{N}$ 
    \begin{align*}
        \sup\limits_{k\geq 0}\left|\left|\frac{1}{n}\mathbb{E}\left(\left(\sum_{j=k+1}^{k+n}\eta_{Y_j}\right)^2\Bigg |\mathcal{F}_{k-r}\right)-\sigma_\eta^2\right|\right|_{L_1}
    \end{align*}
    \begin{align}\label{eq_lemma_cond_mov_block_full_1}
        \leq \sup\limits_{k\geq 0}\left|\left|\frac{1}{n}\mathbb{E}\left(\left(\sum_{j=k+1}^{k+n}\eta_{Y_j}\right)^2\Bigg |\mathcal Y\right)-\sigma_\eta^2\right|\right|_{L_1}+2\sigma_\eta^2P(Y_{r+1}\leq m).
    \end{align}
    In particular
    \begin{align*}
        \sup\limits_{k\geq 0}\left|\left|\frac{1}{n}\mathbb{E}\left(\left(\sum_{j=k+1}^{k+n}\eta_{Y_j}\right)^2\Bigg |\mathcal{F}_{k-r}\right)-\sigma_\eta^2\right|\right|_{L_1}\longrightarrow 0,
    \end{align*}
    as $\min\{r,n\}\to\infty$.
\end{lemma}
\begin{proof}
    From Lemma \ref{lemma_moving_block_var_cond} we already know that 
    \begin{align*}
        \sup\limits_{k\geq 0}\left|\left|\frac{1}{n}\mathbb{E}\left(\left(\sum_{j=k+1}^{k+n}\eta_{Y_j}\right)^2\Bigg |\mathcal Y\right)-\sigma_\eta^2\right|\right|_{L_1}\underset{n\to\infty}{\longrightarrow}0.
    \end{align*}
    Fix now $k\in\mathbb{N}_0$ and $r,n\in\mathbb{N}$. If $k-r\leq 0$ then $\mathcal{F}_{k-r}=\mathcal{Y}$, so (\ref{eq_lemma_cond_mov_block_full_1}) follows immediately. Therefore we only have to look at the case $k-r\geq 1$.\\
    Define the event $D_{k,r}:=\{Y_{k+1}-Y_{k-r}>m\}\in \mathcal{Y}$. On this event it holds $Y_j\geq Y_{k+1}>Y_{k-r}+m$ for $j\geq k+1$. Similar to the proof of $\One_D \mathbb{E}(\eta_{Y_i}|\mathcal{F}_{i-r})=0$ in Lemma \ref{lemma_mixingale_trunc} we can show that
    \begin{align}\label{eq_lemma_cond_mov_block_full_2}
        \One_{D_{k,r}}\mathbb{E}\left(\frac{1}{n}\left(\sum_{j=k+1}^{k+n}\eta_{Y_j}\right)^2\Bigg| \mathcal{F}_{k-r}\right)=\One_{D_{k,r}}\frac{1}{n}\mathbb{E}\left(\left(\sum_{j=k+1}^{k+n}\eta_{Y_j}\right)^2\Bigg| \mathcal{Y}\right).
    \end{align}
    Hence on $D_{k,r}$ we are again in the case of Lemma \ref{lemma_moving_block_var_cond}. \\
    It is clear that $D_{k,r}^c=\{Y_{k+1}-Y_{k-r}\leq m\}\in \mathcal{Y}\subset\mathcal{F}_{k-r}$. We therefore get by the definition of the conditional expectation
    \begin{align*}
        \mathbb{E}\left[\One_{D_{k,r}^c}\mathbb{E}\left(\frac{1}{n}\left(\sum_{j=k+1}^{k+n}\eta_{Y_j}\right)^2\Bigg| \mathcal{F}_{k-r}\right)\right]=\mathbb{E}\left[\One_{D_{k,r}^c}\frac{1}{n}\left(\sum_{j=k+1}^{k+n}\eta_{Y_j}\right)^2\right]
    \end{align*}
    \begin{align*}
        =\mathbb{E}\left[\One_{D_{k,r}^c}\mathbb{E}\left(\frac{1}{n}\left(\sum_{j=k+1}^{k+n}\eta_{Y_j}\right)^2\Bigg| \mathcal{Y}\right)\right].
    \end{align*}
    We estimate
    \begin{align*}
        &\mathbb{E}\left(\One_{D_{k,r}^c}\left|\mathbb{E}\left(\frac{1}{n}\left(\sum_{j=k+1}^{k+n}\eta_{Y_j}\right)^2\Bigg|\mathcal{F}_{k-r}\right)-\sigma_\eta^2\right|\right)\\
        &\leq \mathbb{E}\left(\One_{D_{k,r}^c}\mathbb{E}\left(\frac{1}{n}\left(\sum_{j=k+1}^{k+n}\eta_{Y_j}\right)^2\Bigg|\mathcal{F}_{k-r}\right)\right)+\sigma_\eta^2P(D_{k,r}^c).
    \end{align*}
    The previous equality and the estimate $x\leq |x-y|+y$ for $x,y\geq 0$ gives
    \begin{align*}
        \mathbb{E}\left(\One_{D_{k,r}^c}\left|\mathbb{E}\left(\frac{1}{n}\left(\sum_{j=k+1}^{k+n}\eta_{Y_j}\right)^2\Bigg|\mathcal{F}_{k-r}\right)-\sigma_\eta^2\right|\right)
    \end{align*}
    \begin{align}\label{eq_lemma_cond_mov_block_full_3}
        \leq \mathbb{E}\left(\One_{D_{k,r}^c}\left|\mathbb{E}\left(\frac{1}{n}\left(\sum_{j=k+1}^{k+n}\eta_{Y_j}\right)^2\Bigg|\mathcal{Y}\right)-\sigma_\eta^2\right|\right)+2\sigma_\eta^2P(D_{k,r}^c).
    \end{align}
    Putting (\ref{eq_lemma_cond_mov_block_full_2}) and (\ref{eq_lemma_cond_mov_block_full_3}) together and using that $P(D_{k,r}^c)=P(Y_{r+1}\leq m)$, we get 
    \begin{align*}
        \left|\left|\frac{1}{n}\mathbb{E}\left(\left(\sum_{j=k+1}^{k+n}\eta_{Y_j}\right)^2\Bigg |\mathcal{F}_{k-r}\right)-\sigma_\eta^2\right|\right|_{L_1}
    \end{align*}
    \begin{align*}
        \leq\left|\left|\frac{1}{n}\mathbb{E}\left(\left(\sum_{j=k+1}^{k+n}\eta_{Y_j}\right)^2\Bigg |\mathcal{F}_{k-r}\right)-\sigma_\eta^2\right|\right|_{L_1}+2\sigma_\eta^2P(Y_{r+1}\leq m).
    \end{align*}
    Taking the supremum over $k\in\mathbb{N}_0$ gives (\ref{eq_lemma_cond_mov_block_full_1}).\\
    By Lemma \ref{lemma_moving_block_var_cond} the first term of the inequality (\ref{eq_lemma_cond_mov_block_full_1}) converges to $0$, uniformly in $k$, as $n\to\infty$. The second term tends to $0$, uniformly in $k$, as $r\to\infty$, since $P(Y_{r+1}\leq m)\underset{r\to\infty}{\longrightarrow}0$, which finishes the proof.
\end{proof}

\begin{lemma}\label{lemma_trunc_CLT}
    Let $\sigma_\eta^2$ be defined as in Lemma \ref{prop_c_u}. If $\sigma_\eta^2>0$, then
    \begin{align*}
        \frac{1}{\sqrt{n}}\sum_{i=1}^n \eta_{Y_i}\overset{d}{\underset{n\to\infty}{\longrightarrow}}N(0, \sigma^2_\eta).
    \end{align*}
    If $\sigma_\eta^2=0$, then $\frac{1}{\sqrt{n}}\sum_{i=1}^n \eta_{Y_i}\underset{n\to\infty}{\overset{L_2}{\longrightarrow}}0.$
\end{lemma}
\begin{proof}
    Our goal is to apply Proposition \ref{prop_McLeish} on $((\eta_{Y_i})_{i\in\mathbb{N}}, (\mathcal{F}_j)_{j\in\mathbb{Z}})$. Lemma \ref{lemma_mixingale_trunc} already established that $((\eta_{Y_i})_{i\in\mathbb{N}}, (\mathcal{F}_j)_{j\in\mathbb{Z}})$ is a centered $L_2$-mixingale, whose coefficients $(\theta_r)_{r\in\mathbb{N}_0}$ can be chosen to have size $-1/2$. Also $L_2$-uniform integrability of $\eta_Y$ has already been established. Condition (iii) of Proposition \ref{prop_McLeish} follows from Lemma \ref{lemma_moving_block_cond_full}. So all that is left to prove (when $\sigma_\eta^2>0$) is condition (ii) in Proposition \ref{prop_McLeish}.\\
   We define $S_n:=\sum_{i=1}^n \eta_{Y_i}$.
    Lemma \ref{lemma_moving_block_var_cond} implies for $k=0$
    \begin{align*}
        \left|\left|\frac{1}{n}\mathbb{E}(S_n^2|\mathcal{Y})-\sigma_\eta^2\right|\right|_{L_1}\underset{n\to\infty}{\longrightarrow}0.
    \end{align*}
    Now the tower property yields
    \begin{align*}
        \left|\frac{1}{n}\mathbb{E}(S_n^2)-\sigma_\eta^2\right|=\left|\mathbb{E}\left(\frac{1}{n}\mathbb{E}(S_n^2|\mathcal{Y})-\sigma_\eta^2\right)\right|\leq \mathbb{E}\left(\left|\frac{1}{n}\mathbb{E}( S_n^2|\mathcal{Y})-\sigma_\eta^2\right|\right)\underset{n\to\infty}{\longrightarrow}0.
    \end{align*}
     If $\sigma_\eta^2>0$, then all assumptions of Proposition \ref{prop_McLeish} are satisfied, we therefore get that
    \begin{align*}
        \frac{1}{\sqrt{n}}\sum_{i=1}^n \eta_{Y_i}\overset{d}{\underset{n\to\infty}{\longrightarrow}}N(0, \sigma^2_\eta).
    \end{align*}
    If $\sigma_\eta^2=0$ then the previous calculation gives
    \begin{align*}
        \left|\left|\frac{S_n}{\sqrt{n}}\right|\right|^2_{L_2}=\frac{1}{n}\mathbb{E}(S_n^2)\underset{n\to\infty}{\longrightarrow} 0
    \end{align*}
    and thus $\frac{S_n}{\sqrt{n}}\underset{n\to\infty}{\overset{L_2}{\longrightarrow}}0$.
\end{proof}
\noindent
We now have the main ingredients to prove Theorem \ref{CLT_sampled}
\begin{proof}[Proof of Theorem \ref{CLT_sampled}]
    We first prove (\ref{abs_conv_acf}). In order to do that, we first collect some properties of the process $\xi=(\xi_t)_{t\in\mathbb{R}}$. In Theorem \ref{LLN_sampled} we established that the function $f_X(t)=\mathbb{E}(X_t)$ is almost periodic. Centering by a deterministic almost-periodic function $f_X$ preserves almost periodicity in distributional sense, so $\xi_t=X_t-f_X(t)$ is almost periodic in distributional sense. Note that $m$-dependence of $\xi$ follows immediately from $m$-dependence of $X$. Further, since both $X$ and $(f_X(t))_{t\in\mathbb{R}}$ are $L_2$-uniformly integrable, so is their difference $\xi$.
    Hence also $\xi_Y$ is $L_2$-uniformly integrable, since by independence of $\xi$ and $Y$ we can estimate
    \begin{align*}
        \mathbb{E}(\xi_{Y_i}^2\One_{\{|\xi_{Y_i}|>K\}})=\int_{[0,\infty)}\mathbb{E}(\xi_t^2\One_{\{|\xi_t|>K\}})\, P_{Y_i}(dt)\leq \sup\limits_{t\in\mathbb{R}}\mathbb{E}(\xi_t^2\One_{\{|\xi_t|>K\}}). 
    \end{align*}
    We now come back to the proof of (\ref{abs_conv_acf}). The Cauchy-Schwarz inequality as well as $m$-dependence of $\xi$ yield for $t\in\mathbb{R}$ and $s\geq 0$
    \begin{align*}
        |\mathbb{E}(\xi_t\xi_{t+s})| \leq \sup\limits_{t\in\mathbb{R}}\mathbb{E}(\xi_t^2)\One_{\{s\leq m\}}.
    \end{align*}
    We write using the independent increment property of $Y$ for $i,h\in\mathbb{N}_0$
    \begin{align*}
        \mathbb{E}(\xi_{Y_i}\xi_{Y_{i+h}})=\int_{[0,\infty)}\int_{[0,\infty)}\mathbb{E}(\xi_t\xi_{t+s})P_{Y_h}(ds)P_{Y_i}(dt),
    \end{align*}
    which shows, paired with the previous estimate,
    \begin{align}\label{ineq_conv_acf_1}
        |\mathbb{E}(\xi_{Y_i}\xi_{Y_{i+h}})|\leq \sup\limits_{t\in\mathbb{R}}\mathbb{E}(\xi_t^2) P(Y_h\leq m).
    \end{align}
    Note that since $P(Y_h\leq m)$ decays exponentially as $h\to\infty$, $|\mathbb{E}(\xi_{Y_i}\xi_{Y_{i+h}})|$ is absolutely summable. Let $Z$ denote the stationary limit process given by Theorem \ref{Theorem_stat_Z} applied on the process $\xi$. Since $\xi_Y$ is $L_2$-uniformly integrable, we get, using Remark \ref{remark_after_stat_thrm} (iii) and Lemma \ref{prop_c_u}(iv) that 
    \begin{align*}
        \mathbb{E}(\xi_{Y_i}\xi_{Y_{i+h}})\underset{i\to\infty}{\longrightarrow}\mathbb{E}(Z_0Z_h)=\gamma_Z(h).
    \end{align*}
    and consequently
    \begin{align*}
        |\gamma_Z(h)|\leq \sup\limits_{t\in\mathbb{R}}\mathbb{E}(\xi_t^2) P(Y_h\leq m).
    \end{align*}
    This gives 
    \begin{align*}
        \sum_{h=1}^\infty |\gamma_Z(h)|\leq \sup\limits_{t\in\mathbb{R}}\mathbb{E}(\xi_t^2)  \sum_{h=1}^\infty P(Y_h\leq m)<\infty,
    \end{align*}
    which proves the first part of (\ref{abs_conv_acf}). For the second part of (\ref{abs_conv_acf}) and the statement of the theorem when $\sigma^2=0$, we expand
    \begin{align*}
        \frac{1}{n}\text{Var}\left(\sum_{i=1}^n\xi_{Y_i}\right)=\frac{1}{n}\sum_{i=1}^n\mathbb{E}(\xi_{Y_i}^2)+2\sum_{h=1}^{n-1}\frac{1}{n}\sum_{i=1}^{n-h}\mathbb{E}(\xi_{Y_i}\xi_{Y_{i+h}}).
    \end{align*}
    The Cesàro Lemma gives that for each fixed $h\in\mathbb{N}_0$
    \begin{align*}
        \frac{1}{n}\sum_{i=1}^{n-h}\mathbb{E}(\xi_{Y_i}\xi_{Y_{i+h}})\underset{n\to\infty}{\longrightarrow}\gamma_{Z}(h).
    \end{align*}
    Since 
    \begin{align*}
        \left|\frac{1}{n}\sum_{i=1}^{n-h}\mathbb{E}(\xi_{Y_i}\xi_{Y_{i+h}})\right|\leq \sup\limits_{t\in\mathbb{R}}\mathbb{E}(\xi_t^2)P(Y_h\leq m)
    \end{align*}
    and the right-hand side is absolutely summable, we obtain by Lebesgue's dominated convergence theorem
    \begin{align*}
        \frac{1}{n}\text{Var}\left(\sum_{i=1}^n\xi_{Y_i}\right)\underset{n\to\infty}{\longrightarrow}\gamma_{Z}(0)+2\sum_{h=1}^\infty\gamma_{Z}(h)=:\sigma^2
    \end{align*}
    and in particular $\sigma^2\geq 0$, finishing the proof of (\ref{abs_conv_acf}).\\ 
    Similar to Lemma \ref{lemma_trunc_CLT} the case $\sigma^2=0$ can be immediately deduced from our work above. Indeed we have
    \begin{align*}
        \left|\left|\frac{1}{\sqrt{n}}\sum_{i=1}^n \xi_{Y_i}\right|\right|_{L_2}^2=\frac{1}{n}\text{Var}\left(\sum_{i=1}^n \xi_{Y_i}\right)\underset{n\to\infty}{\longrightarrow}\sigma^2=0,
    \end{align*}
    so $\frac{\sum_{i=1}^n\xi_{Y_i}}{\sqrt{n}}\underset{n\to\infty}{\overset{L_2}{\longrightarrow}}0$. For the remainder of the proof we therefore assume $\sigma^2>0$.\\
    The idea for the remaining part of the proof is to prove the asymptotic result (i) with a truncation argument. With Lemma \ref{lemma_trunc_CLT} we already have the main ingredient available.\\
    We define for $N>0$ the truncation function 
    \begin{align*}
        h_N(t):=\begin{cases}
            t & \text{ if } |t|\leq N,\\
            N & \text{ if } t>N,\\
            -N & \text{ if } t < -N,
        \end{cases}
    \end{align*}
    and denote $X_t^{(N)}:=h_N(X_t), f_X^{(N)}(t):=\mathbb{E}(X_t^{(N)})$ and $\xi_t^{(N)}:=h_N(X_t)-f_X^{(N)}(t)$. Since $h_N$ is bounded and continuous we get from \cite[Proposition 2.8]{BM-22} that the process $(h_N(X_t))_{t\in\mathbb{R}}$ is almost periodic in distributional sense and since $f_X^{(N)}$ is Bohr almost periodic, we obtain that the process $(\xi_t^{(N)})_{t\in\mathbb{R}}$ is almost periodic in distributional sense. It is also clear that the truncation does not affect the $m$-dependence property, so $(\xi_t^{(N)})_{t\in\mathbb{R}}$ is an $m$-dependent process. Since $\xi^{(N)}_t$ is clearly bounded by $2N$ we get that $(\xi_t^{(N)})_{t\in\mathbb{R}}$ is $L_p$-uniformly integrable for every $p>0$. Similar to before we can also infer $L_p$-uniform integrability of $\xi_Y^{(N)}$.\\ Let furthermore $Z^{(N)}$ be the stationary limit process given by Theorem \ref{Theorem_stat_Z} applied on the process $\xi^{(N)}$ and put 
    \begin{align*}
        \sigma_N^2:=\gamma_{Z^{(N)}}(0)+2\sum_{h=1}^\infty \gamma_{Z^{(N)}}(h),
    \end{align*}
    with $\gamma_{Z^{(N)}}$ denoting the autocovariance function of $Z^{(N)}$.
    Similar to the proof of (\ref{abs_conv_acf}) above one can show that this sum is absolutely convergent and that $\frac{1}{n}\text{Var}\left(\sum_{i=1}^n \xi_{Y_i}^{(N)}\right)\underset{n\to\infty}{\longrightarrow}\sigma_N^2$ for each fixed $N>0$.\\
    Taking $K=2N$ and $\eta=\xi^{(N)}$, we see that all assumptions of Lemma \ref{lemma_trunc_CLT} are satisfied and thus
    \begin{align*}
        \frac{1}{\sqrt{n}}\sum_{i=1}^n \xi^{(N)}_{Y_i}\overset{d}{\underset{n\to\infty}{\longrightarrow}}N(0, \sigma^2_N),
    \end{align*}
    if $\sigma_N^2>0$ and $\frac{1}{\sqrt{n}}\sum_{i=1}^n \xi^{(N)}_{Y_i}\overset{L_2}{\underset{n\to\infty}{\longrightarrow}}0$ if $\sigma_N^2=0$.\\
    Define $\chi_t^{(N)}:=\xi_t-\xi_t^{(N)}$. Our next goal is to show that 
    \begin{align*}
        \sup\limits_{n\in\mathbb{N}}\left|\left|\frac{1}{\sqrt{n}}\sum_{i=1}^n \chi_{Y_i}^{(N)}\right|\right|_{L_2}^2\underset{N\to\infty}{\longrightarrow}0.
    \end{align*}
    We have $\mathbb{E}(\chi_{Y_i}^{(N)})=0$ for all $i\in\mathbb{Z}$ and $N>0$. Furthermore we can estimate
    \begin{align*}
        &\mathbb{E}((\chi_{Y_i}^{(N)})^2|Y_i)=\text{Var}(X_{Y_i}-X_{Y_i}^{(N)}|Y_i)\leq \mathbb{E}((X_{Y_i}-X_{Y_i}^{(N)})^2|Y_i)
    \end{align*}
    and thus
    \begin{align*}
        \mathbb{E}((\chi_{Y_i}^{(N)})^2)\leq \mathbb{E}((X_{Y_i}-X_{Y_i}^{(N)})^2)\leq \sup\limits_{t\in\mathbb{R}}\mathbb{E}((X_t-X_t^{(N)})^2)\leq \sup\limits_{t\in\mathbb{R}}\mathbb{E}(X_t^2 \One_{\{|X_t|>N\}}),
    \end{align*}
    which converges to $0$ as $N\to\infty$, since $(X_t)_{t\in\mathbb{R}}$ is $L_2$-uniformly integrable.\\
    Similar to (\ref{ineq_conv_acf_1}) we get an estimate for $|\mathbb{E}(\chi_{Y_k}^{(N)}\chi_{Y_l}^{(N)})|$. Let $i>l$, then
    \begin{align*}
        |\mathbb{E}(\chi_{Y_i}^{(N)}\chi_{Y_l}^{(N)})|\leq \sup\limits_{t\in\mathbb R}\mathbb{E}((\chi_t^{(N)})^2)P(Y_{i-l}\leq m)
        \leq \sup\limits_{t\in\mathbb{R}}\mathbb{E}(X_t^2 \One_{\{|X_t|>N\}})P(Y_{i-l}\leq m).
    \end{align*}
    Now let $n\in\mathbb{N}$, then
    \begin{align*}
        &\mathbb{E}\left(\left(\sum_{i=1}^n \chi_{Y_i}^{(N)}\right)^2\right)=\sum_{i=1}^n\mathbb{E}((\chi_{Y_i}^{(N)})^2)+2\sum_{i=1}^{n}\sum_{l=1}^{i-1}\mathbb{E}(\chi_{Y_i}^{(N)}\chi_{Y_l}^{(N)})\\
        \leq & n \sup\limits_{t\in\mathbb{R}}\mathbb{E}(X_t^2 \One_{\{|X_t|>N\}})+2 \sup\limits_{t\in\mathbb{R}}\mathbb{E}(X_t^2 \One_{\{|X_t|>N\}})\sum_{h=1}^{n-1}(n-h)P(Y_{h}\leq m)\\
        &\leq n \sup\limits_{t\in\mathbb{R}}\mathbb{E}(X_t^2 \One_{\{|X_t|>N\}}) \left(1+2\sum_{h=1}^{\infty}P(Y_h\leq m)\right) \underset{N\to\infty}{\longrightarrow} 0.
    \end{align*}
    Thus 
    \begin{align*}
        \sup\limits_{n\in\mathbb{N}}\left|\left|\frac{1}{\sqrt{n}}\sum_{i=1}^n\chi_{Y_i}^{(N)}\right|\right|_{L_2}^2\leq \sup\limits_{t\in\mathbb{R}}\mathbb{E}(X_t^2 \One_{\{|X_t|>N\}}) \left(1+2\sum_{h=1}^{\infty}P(Y_h\leq m)\right) \underset{N\to\infty}{\longrightarrow}0.
    \end{align*}
    The above estimate of the $L_2$ average of the sum of the truncations also lets us conclude that $\sigma_N^2\underset{N\to\infty}{\longrightarrow}\sigma^2$. Indeed, let $N,n\in\mathbb{N}$, then we get by the reverse triangle inequality
    \begin{align*}
        &\left|\frac{1}{\sqrt{n}}\sqrt{\text{Var}\left(\sum_{i=1}^n\xi_{Y_i}^{(N)}\right)}-\frac{1}{\sqrt{n}}\sqrt{\text{Var}\left(\sum_{i=1}^n\xi_{Y_i}\right)} \right|=\frac{1}{\sqrt{n}}\left|\, \left|\left|\sum_{i=1}^n\xi_{Y_i}^{(N)}\right|\right|_{L_2}-\left|\left|\sum_{i=1}^n\xi_{Y_i}\right|\right|_{L_2}\right|\\
        &\leq \frac{1}{\sqrt{n}}\left|\left|\sum_{i=1}^n \chi_{Y_i}^{(N)}\right|\right|_{L_2}\leq \sqrt{\sup\limits_{t\in\mathbb{R}}\mathbb{E}(X_t^2 \One_{\{|X_t|>N\}}) \left(1+2\sum_{h=1}^{\infty}P(Y_h\leq m)\right)}.
    \end{align*}
    Fixing $N\in\mathbb{N}$ and letting $n\to\infty$ gives
    \begin{align*}
        \left|\sqrt{\sigma_N^2}-\sqrt{\sigma^2}\right|\leq \sqrt{\sup\limits_{t\in\mathbb{R}}\mathbb{E}(X_t^2 \One_{\{|X_t|>N\}}) \left(1+2\sum_{h=1}^{\infty}P(Y_h\leq m)\right)}.
    \end{align*}
    If we now let $N\to\infty$, we obtain $\sigma_N^2\underset{N\to\infty}{\longrightarrow}\sigma^2$. This also implies that in the case of $\sigma^2>0$ we know that $\sigma_N^2>0$ for $N$ large enough.\\
    Putting everything together, we get
    \begin{align*}
        &\frac{1}{\sqrt{n\sigma^2}}\sum_{i=1}^n\xi_{Y_i}^{(N)}\underset{n\to\infty}{\overset{d}{\longrightarrow}}N\left(0,\frac{\sigma^{2}_N}{\sigma^2}\right) \qquad \forall N \text{ large},\\
        &N\left(0,\frac{\sigma^{2}_N}{\sigma^2}\right)\underset{N\to\infty}{\overset{d}{\longrightarrow}}N(0,1),\\
        &\sup\limits_{n\in\mathbb{N}}\left|\left|\frac{1}{\sqrt{n}}\sum_{i=1}^n\chi_{Y_i}^{(N)}\right|\right|_{L_2}^2 \underset{N\to\infty}{\longrightarrow}0.
    \end{align*}
    This yields for $\varepsilon>0$, using Chebyshev's inequality
    \begin{align*}
        &\lim\limits_{N\to\infty}\limsup\limits_{n\to\infty}P\left(\left|\frac{1}{\sqrt{n\sigma^2}}\sum_{i=1}^n\xi_{Y_i}^{(N)}-\frac{1}{\sqrt{n\sigma^2}}\sum_{i=1}^n\xi_{Y_i}\right|>\varepsilon\right)\\
        \leq& \frac{1}{\varepsilon^2\sigma^2}\lim\limits_{N\to\infty} \sup\limits_{n\in\mathbb{N}} \left|\left|\frac{1}{\sqrt{n}}\sum_{i=1}^n\chi_{Y_i}^{(N)}\right|\right|_{L_2}^2=0.
    \end{align*}
    Using the generalised Lemma of Slutsky (see e.g. \cite[Theorem 6.3.9]{BrockwellDavis1991}) gives us that   
    \begin{align*}
        \frac{1}{\sqrt{n\sigma^2}}\sum_{k=1}^n\xi_{Y_k}\underset{n\to\infty}{\longrightarrow}N(0,1),
    \end{align*}
    which finishes the proof in the positive time case.\\
    For the negative time case we define again $\widetilde{X}$ as in the proof of Theorem \ref{LLN_sampled}. Thus $X_{Y_{-k}}-f_X(Y_{-k})\overset{a.s.}{=}X_{-Y'_k}-f_X(-Y'_k)=\widetilde{X}_{Y'_k}-f_{\widetilde{X}}(Y_k')$. It is again clear that the assumptions of Theorem \ref{CLT_sampled} are satisfied for $\widetilde{X}$ and $Y'$. All that is left to prove is now that $\gamma_Z=\gamma_{\widetilde{Z}}$, with $\widetilde{Z}$ being the stationary limit process from Theorem \ref{Theorem_stat_Z} applied to $\widetilde{\xi}_t:=\widetilde{X}_t-f_{\widetilde{X}}(t)$. \\
    With Theorem \ref{Theorem_stat_Z} we know that
    \begin{align*}
        (\xi_{Y_{-g}}, \xi_{Y_{-(g+h)}})\underset{g\to\infty}{\overset{d}{\longrightarrow}}(Z_h,Z_0) \text{ as well as }(\widetilde{\xi}_{Y_{g}'}, \widetilde{\xi}_{Y_{(g+h)}'})\underset{g\to\infty}{\overset{d}{\longrightarrow}}(\widetilde{Z}_0,\widetilde{Z}_h).
    \end{align*}
    Since $(\widetilde{\xi}_{Y_{g}'}, \widetilde{\xi}_{Y_{(g+h)}'})=(\xi_{Y_{-g}}, \xi_{Y_{-(g+h)}})$, we get $(Z_h,Z_0)\overset{d}{=}(\widetilde{Z}_0, \widetilde{Z}_h)$, which gives $\gamma_Z=\gamma_{\widetilde{Z}}$.
\end{proof}
\noindent
Let $X=(X_n)_{n\in\mathbb{N}}$ be a sequence of random variables on the probability space $(\Omega, \mathcal{F}, P)$ and let $\mathcal{G}_1, \mathcal{G}_2 \subset \mathcal{F}$ be sub-$\sigma$-algebras of $\mathcal{F}$. Define $\alpha(\mathcal{G}_1,\mathcal{G}_2)=\sup\{|P(G_1\cap G_2)-P(G_1)P(G_2)|: G_1\in \mathcal{G}_1, G_2\in\mathcal{G}_2\}$. The sequence $X$ is called \emph{strongly mixing}, if 
\begin{align*}
    \alpha_X(h):=\sup\limits_{n\in\mathbb{N}}\alpha(\sigma(X_j: j\leq n), \sigma(X_j:j\geq n+h))\underset{h\to\infty}{\longrightarrow}0
\end{align*} 
In \cite[Proposition 2.2]{BrandesCurato2019} it was shown that \emph{strictly stationary} $m$-dependent processes are strongly mixing with exponentially decaying mixing coefficients, after random sampling. This result is then used to prove, among other results, that a central limit theorem holds for the sampled process (see \cite[Theorem 3.2]{BrandesCurato2019}). Since plenty of central limits theorems for strongly mixing non-stationary sequences exist (see e.g \cite{Herrndorf1985}), one might be inclined to use a similar route in order to prove Theorem \ref{CLT_sampled}. However in the case of Theorem \ref{CLT_sampled}, the sampled process $X_Y$ does not need to be strongly mixing, as the following example will demonstrate.
\begin{example}
    Define for $t\in\mathbb{R}$
    \begin{align*}
        X_t:= \cos(2\pi t)Z_t, \qquad Z_t=\begin{cases}
            1 & \text{ if } W_{t+1}-W_t\geq 0,\\
            -1 & \text{ if } W_{t+1}-W_t<0,
        \end{cases}
    \end{align*}
    with $(W_t)_{t\in\mathbb{R}}$ being a standard Brownian motion. We further define for $k\in\mathbb{N}$
    \begin{align*}
        Y_k=\sum_{i=1}^k\psi_i, \qquad P(\psi_i=2)=P(\psi_i=3)=P(\psi_i=2+\sqrt{2})=\frac{1}{3} \quad \forall i\in\mathbb{N},
    \end{align*}
    with $(\psi_i)_{i\in\mathbb{N}}$ being an independent sequence, chosen independent of $X$.
    Since $Z=(Z_t)_{t\in\mathbb{R}}$ is strictly stationary and only depends on the increments of $[t,t+1]$, we immediately get that $X=(X_t)_{t\in\mathbb{R}}$ is $1$-dependent. The finite dimensional distributions of $X$ are furthermore $1$-periodic, so $X$ is almost periodic. Since $\mathbb{E}(Z_t)=0$, we get $\mathbb{E}(X_t)=0$ and since $X$ is obviously bounded, it is also $L_2$-uniformly integrable. $X$ therefore satisfies all assumptions of Theorem \ref{CLT_sampled}. \\
    Since $Y$ obviously induces a random sampling in the case $\mathbb{G}=\mathbb{Z}$ and since $\mathcal{L}(Y_1)$ is not a lattice distribution, we get (see the comment after (\ref{eq_lattice_sampling})) that $|\widehat{\mathcal{L}}(Y_1)(z)|\ne1$ for all $z\ne 0$. Hence $Y$ satisfies all assumptions of Theorem \ref{CLT_sampled}.\\ 
    We now prove $\alpha_{X_Y}(h)\geq \frac{1}{4}$ for all $h\in\mathbb{N}$, with $\alpha_{X_Y}$ being the strong mixing coefficient of $X_Y$. Define $M_k:=\sum_{i=1}^k\One_{\{\psi_i=3\}}$ and $N_k:=\sum_{i=1}^k\One_{\{\psi_i=2+\sqrt2\}}$. Then 
    \begin{align*}
        Y_k=2k+M_k+\sqrt{2}N_k
    \end{align*}
    and thus $X_{Y_k}^2=\cos^2(2\pi Y_k)=\cos^2(2\pi \sqrt{2}N_k)$, since $2k+M_k$ is a natural number. Furthermore the map $\mathbb{N}_0 \ni n\mapsto\cos^2(2\pi \sqrt{2}n)$ is injective. To see this assume $\cos^2(2\pi \sqrt{2}n)=\cos^2(2\pi \sqrt{2}m)$. Using the identity $2\cos^2(x)-1=\cos(2x)$ we get $\cos(4\pi \sqrt{2} n)=\cos(4\pi \sqrt{2} m)$, which yields $2\sqrt{2}(n+m)\in\mathbb{Z}$ or $2\sqrt{2}(n-m)\in\mathbb{Z}$. In both cases we obtain $n=m$. \\
    Now for $n\in\mathbb{N}_0$, $\{N_k=n\}=\{X_{Y_k}^2=\cos^2(2\pi\sqrt{2}n)\}$, which yields $\sigma(N_k)\subset \sigma(X_{Y_k})$. Define now for $h,k\in\mathbb{N}$
    \begin{align*}
        A_k:=\left\{N_k\leq \frac{k}{3}\right\}, \qquad B_{k,h}:=\left\{N_{k+h}\leq \frac{k}{3}\right\},
    \end{align*}
    then $A_k\in \sigma(X_{Y_k})\subset \sigma(X_{X_{Y_j}}:j\leq k)$, $B_{k,h}\in \sigma(X_{Y_{k+h}})\subset\sigma(X_{Y_j}:j\geq k+h)$ and $B_{k,h}\subset A_k$. Since $N_k\sim \text{Bin}\left(k, \frac{1}{3}\right)$, we get with the central limit theorem $\frac{N_k-\frac{k}{3}}{\sqrt{\frac{2k}{9}}}\underset{k\to\infty}{\overset{d}{\longrightarrow}}N(0,1)$. Thus
    \begin{align*}
        P(A_k)=P\left(\frac{N_k-\frac{k}{3}}{\sqrt{\frac{2k}{9}}}\leq 0\right) \underset{k\to\infty}{\longrightarrow}\frac{1}{2}
    \end{align*}
    and
    \begin{align*}
        P(B_{k,h})=P\left(\frac{N_{k+h}-\frac{k+h}{3}}{\sqrt{\frac{2(k+h)}{9}}}\leq \frac{-\frac{h}{3}}{\sqrt{\frac{2(k+h)}{9}}}\right) \underset{k\to\infty}{\longrightarrow}\frac{1}{2}.
    \end{align*}
    We can therefore estimate for $h\in\mathbb{N}$
    \begin{align*}
        \alpha_{X_Y}(h)\geq |P(A_k\cap B_{k,h})-P(A_k)P(B_{k,h})|=|P(B_{k,h})-P(A_k)P(B_{k,h})|\underset{k\to\infty}{\longrightarrow}\frac{1}{4}.
    \end{align*}
    The sequence $(X_{Y_k})_{k\in\mathbb{N}}$ can therefore not be strongly mixing.
\end{example}
\section*{Acknowledgements}
The authors are grateful to Alexander Lindner for his support, insightful comments, and many fruitful discussions, which have substantially improved both the content and the presentation of this paper.

\end{document}